\documentclass[11pt,leqno,twoside]{amsart}
 \usepackage[a4paper,left=20mm,right=20mm,top=25mm,bottom=25mm]{geometry} 
\usepackage{amsmath,amssymb,amsthm,enumerate,hyperref,color}
\numberwithin{equation}{section}
\newtheorem{theorem}{Theorem}[section]

\newtheorem{proposition}[theorem]{Proposition}
\newtheorem{lemma}[theorem]{Lemma}
\newtheorem{corollary}[theorem]{Corollary}
\newtheorem{remark}[theorem]{Remark}
\usepackage{enumerate}

\newcommand{\nor}{\Arrowvert}

\def\e{{\varepsilon}}

\def\na{\nabla} 
\def\d{\delta}

\def\G{{\Gamma}}

\def\L{{\Lambda}}
\def\l{{\lambda}}
\def\a{{\alpha}}

\def\de{\partial}

\newcommand{\R}{\mathbb{R}}
\newcommand{\N}{\mathbb{N}}

\providecommand{\norm}[1]{\lVert#1\rVert}
\newcommand{\rad}{{\text{\upshape rad}}}
\newcommand{\even}{{\text{\upshape even}}}

\def\sideremark#1{\ifvmode\leavevmode\fi\vadjust{\vbox to0pt{\vss
 \hbox to 0pt{\hskip\hsize\hskip1em
 \vbox{\hsize2.1cm\tiny\raggedright\pretolerance10000
  \noindent #1\hfill}\hss}\vbox to15pt{\vfil}\vss}}}%

\begin{document}
\title[Hardy-Sobolev equation]{Bifurcation analysis of the Hardy-Sobolev equation}

\thanks{D. Bonheure \& J.B. Casteras were supported by MIS F.4508.14 (FNRS) and PDR T.1110.14F (FNRS); J. B. Casteras is supported by FCT - Funda\c c\~ao para a Ci\^encia e a Tecnologia, under the project: UIDB/04561/2020
J.B. Casteras would like to thank the Belgian Fonds de la Recherche Scientifique -- FNRS;
F. Gladiali is supported by FABBR-2017 and partially supported by Prin-2015KB9WPT and Gruppo Nazionale per l'Analisi Matematica, la Probabilit\`a e le loro Applicazioni (GNAMPA) of the Istituto Nazionale di Alta Matematica (INdAM)} 
\author{Denis Bonheure} 
\author{Jean-Baptiste Casteras}
\author{Francesca Gladiali}

\address{Denis Bonheure, 
\newline \indent D\'epartement de Math\'ematiques, Universit\'e Libre de Bruxelles,
\newline \indent CP 214, Boulevard du triomphe, B-1050 Bruxelles, Belgium.}
\email{Denis.Bonheure@ulb.be}

\address{Jean-Baptiste Casteras
\newline \indent  CMAFCIO, Faculdade de Ci\^encias da Universidade de Lisboa, \indent 
\newline \indent Edificio C6, Piso 1, Campo Grande 1749-016 Lisboa, Portugal .\indent }
\email{jeanbaptiste.casteras@gmail.com}

\address{Francesca Gladiali
\newline \indent Dipartimento di Chimica e farmacia, Universit\`a di Sassari\indent 
\newline \indent Via Piandanna 4, 07100 Sassari, Italy.\indent }
\email{fgladiali@uniss.it}

\begin{abstract} In this paper, we prove existence of multiple non-radial solutions to the Hardy-Sobolev equation 
\begin{equation*}
\begin{cases}
-\Delta u-\displaystyle\frac \gamma{|x|^2}u=\displaystyle\frac{1}{|x|^s}|u|^{p_s-2}u & \text{ in } \R^N\setminus\{0\},\\
u\geq 0, & 
\end{cases}\end{equation*}
where $N\geq 3$, $s\in[0,2)$, $p_s=\frac{2(N-s)}{N-2}$ and $\gamma\in (-\infty,\frac{(N-2)^2} 4)$. We extend results of E.N. Dancer, F.~Gladiali, M.~Grossi, Proc. Roy. Soc. Edinburgh Sect. A 147 (2017) where only the case $s=0$ is considered. Moreover, thanks to monotonicity properties of the solutions, we separate two branches of non-radial solutions. 
\end{abstract}

\maketitle



\tableofcontents

\section{Introduction}
Lots of works have been devoted to the study of the best constant $C_{HS}$ in the Hardy-Sobolev inequality
\begin{equation}\label{HS-ineq}
\int_{\Omega} |\nabla u|^2\, dx - \gamma\int_{\Omega} \frac{u^2}{|x|^2}\, dx\ge C_{HS} \left(\int_\Omega \dfrac{|u|^{p_s}}{|x|^s} \, dx\right)^{\frac{2}{p_s}},\ u\in C_0^\infty (\Omega).
\end{equation}
This inequality holds true for any regular domain $\Omega\subseteq \R^N$ in dimension $N\ge 3$ and for 
\begin{equation}\label{eq:assum-param-s-gamma}
s\in[0,2) \text{ and }\gamma\in (-\infty,\frac{(N-2)^2} 4)
\end{equation}
with
\begin{equation}\label{eq:def-p_s}
p_s=\frac{2(N-s)}{N-2}.
\end{equation}
We refer to the survey paper \cite{GhRo4} as an entry to the extensive related literature.  

It is well-known that if $\Omega$ is a bounded domain having $0$ in its interior, then $C_{HS}$ is never achieved and, as a consequence, there is no energy minimizing solutions to the associated Euler-Lagrange equation completed by Dirichlet boundary condition. The situation drastically changes if $0$ is on the boundary of $\Omega$ as shown by Ghoussoub and Robert \cite{GR-interior,GhRo,GhRo3}. 

When $\Omega=\R^N$, the best constant $C_{HS}$ is achieved if and only if $\{s >0\}$ or $\{s =0 \text{ and } \gamma \ge 0\}$, see \cite{GhRo4}.
The Hardy-Sobolev inequality \eqref{HS-ineq} is a family of interpolation inequalities between the limit cases $s=2$ which yields Hardy inequality
$$\frac{(N-2)^2}{4}\int_{\R^N} \frac{u^2}{|x|^2}\, dx\leq \int_{\R^N} |\nabla u|^2 \, dx,\ u\in C_0^\infty (\R^N),$$
and Sobolev inequality recovered when $s=0$.
As noticed e.g. in~\cite{DELT09}, it is also equivalent to the celebrated Caffarelli-Kohn-Nirenberg inequality \cite{Il,GMGT,CKN} 
\begin{equation}\label{ineq:CKN}
\left(\int_{\R^N} |x|^{-bp} |u|^p \, dx \right)^{\frac{2}{p}}\leq C \int_{\R^N}|x|^{-2a}|\nabla u|^2\, dx,\ u\in C_0^\infty (\R^N), 
\end{equation}
where
$$-\infty < a < \frac{N-2}{2},\ 0\leq b-a \leq 1, \text{ and } p=\frac{2N}{N-2+2(b-a)}.$$
Indeed, setting $w(x)=|x|^{-a}u(x)$, we have
$$\dfrac{\int_{\Omega} |x|^{-2a} |\nabla u|^2\, dx}{\left(\int_\Omega |x|^{-bp} |u|^p\, dx \right)^{\frac{2}{p}}}=\dfrac{\int_{\Omega} |\nabla w|^2\, dx -\gamma \int_{\Omega} {|x|^{-2}}{w^2}\, dx}{\left( \int_{\Omega}{|x|^{-s}}{|w|^{p_s}} \, dx\right)^{\frac{2}{p_s}}},$$
for $s=(b-a)p$ and $\gamma =a(N-2-a)$. 

Dolbeault, Esteban and Loss \cite{DEL-16} recently proved an optimal rigidity result for the Euler-Lagrange equation associated to \eqref{ineq:CKN}
\begin{equation}\label{eq:CKN}
-\mathrm{div}\left(|x|^{-2a}\nabla u\right)=|x|^{-bp} u^{p-1} \  \text{ in }\R^N\setminus\{0\}
\end{equation}
in the range $2<p<\frac{2N}{N-2}$. 
Namely, assuming the integrability condition 
\begin{equation}\label{eq:cond-integral}
\int_{\R^N}  |x|^{-bp} u^{p}\, dx<~\infty, 
\end{equation}
they showed this equation has a unique (therefore radial) nonnegative solution  whenever 
\begin{equation}\label{eq:cond-DEL}
0\le a<\frac{N-2}2 \ \text{ and } b>0 
\text{ or } a<0 \text{ and }
b\ge b_{FS}(a)\end{equation}
with
\[ b_{FS}(a)=\frac{N\left(\frac{N-2} 2-a\right)}{2\sqrt{\left(\frac{N-2} 2-a\right)^2+N-1 }} +a- \frac{N-2} 2.  \]
When $a<0$ and $b<b_{FS}(a)$, Felli and Schneider \cite{FS} have previously shown that the best constant in \eqref{ineq:CKN} is achieved by non radial functions only and, as a byproduct, \eqref{eq:CKN} has non-radial nonnegative solutions and uniqueness is broken.

We precisely address in this paper the question of existence of non-radial nonnegative solutions of the Hardy-Sobolev equation
\begin{equation}\label{problem-intro}
\begin{cases}
-\Delta u-\displaystyle\frac \gamma{|x|^2}u=\displaystyle\frac{1}{|x|^s}|u|^{p_s-2}u & \text{ in } \R^N\setminus\{0\},\\
u\geq 0, & 
\end{cases}\end{equation}
which is the Euler-Lagrange equation associated to \eqref{HS-ineq}. We restrict our attention to solutions in the Sobolev space
$$D^{1,2}\left(\R^N\right)=\{u\in L^{2^*}(\R^N):  |\na u| \in L^2(\R^N)\}.$$
Setting again $w(x)=|x|^{-a}u(x)$, we see that solutions of \eqref{eq:CKN} satisfying the integrability condition \eqref{eq:cond-integral} correspond to solutions of \eqref{problem-intro} that satisfy 
$$\int_{\R^N}  |x|^{-s} w^{p_s}\, dx<\infty.$$ 
Applying the rigidity result of \cite{DEL-16} proved for \eqref{eq:CKN} to \eqref{problem-intro} yields uniqueness for $s>0$ and 
$$\gamma \ge \gamma_{FS} := \frac{(N-2)^2}4\left(\frac{(N-s)^2-N^2}{(N-s)^2-(N-2)^2}\right),$$
whereas the symmetry breaking results of \cite{FS} gives non-radial solutions for $\gamma < \gamma_{FS}$.
The case $s=0$ has been treated by Terracini \cite{T}. She showed, among other things, that, when $\gamma \in [0, \frac{(N-2)^2}{4})$, the problem \eqref{problem-intro} has a unique (therefore radial) solution, up to rescaling,  whereas for some $\gamma<\gamma^*<0$, there are two solutions : one is radial and the other is not. Motivated by this result, Dancer, Gladiali and Grossi \cite{DGG} computed the Morse index of the radial solution for $\gamma<0$ and the kernel of the linearized operator at the degeneracy points $\gamma_j$, see \eqref{eq:gamma_j} below with $s=0$. This analysis yields existence of continua of non-radial solutions bifurcating from the radial one at the critical values $\gamma_j$.

Our goal in this paper is to consider the case $s\ge 0$. We will not only generalize and extend \cite{DGG}, but also go deeper into the bifurcation analysis. 
Improving arguments from \cite{G}, we prove monotonicity properties of the solutions along two branches of non-radial solutions.  This allow us to separate them, see Section \ref{se:4} for more details. 

Our first crucial observation is a one-to-one correspondence between radial solutions to \eqref{problem-intro} and the following ODE
\begin{equation}
\label{one-to-onecorre}
  \begin{cases}
    -\left(r^{q_s-1}v'\right)'=  D_\gamma
    r^{q_s-1}
    v^{\frac{q_s+2}{q_s-2}} & \text{ for } r \in(0,\infty)\\
    v\geq 0, \quad 
    \int_0^{\infty}r^{q_s-1}\left(v'(r)\right)^2\ dr <\infty
  \end{cases}
\end{equation}
where $q_s:=\frac {2(N-s)}{2-s}$ and $D_\gamma=\frac {4(N-2)^2}{(2-s)^2((N-2)^2-4\gamma)}$. We refer to Lemma \ref{lem:equiv-rad} for a precise statement. As a direct consequence of this fact, we deduce that radial solutions to \eqref{problem-intro} are given by
  \begin{equation}\label{sol-rad}
    u_{\gamma,\lambda}(x):=\frac{ \l^{\frac{N-2}2\nu_\gamma} 
    |x|^{\frac{N-2}2(\nu_\gamma-1)}\big[(N-s)(N-2)\nu_\gamma^2\big]^{\frac{N-2}{2(2-s)} }}{\left(1+\l^{(2-s)\nu_\gamma}|x|^ {(2-s)\nu_\gamma}\right)^{\frac {N-2}{2-s}}},
  \end{equation}
  where $\l>0$ and 
\begin{equation}\label{nu-gamma}
\nu_{\gamma}:=\sqrt{1-\frac {4\gamma} {(N-2)^2}}.
\end{equation}
This correspondence also gives information on the linearization of equation \eqref{problem-intro} at $u_{\gamma,\lambda}$. 
We recall that the Morse index of $u_{\gamma,\lambda}$ is the maximal dimension of a subspace of $D^{1,2}(\R^N)$ where the quadratic form corresponding to the linearized operator
\begin{equation}
\label{introeq:linearized2}
L_{\gamma,\lambda} v :=-\Delta v - \frac \gamma{|x|^2}v-(p_s -1)\frac{|u_{\gamma,\lambda}|^{p_s-2}v}{|x|^s},\ v\in D^{1,2}(\R^N)
\end{equation}
is negative definite. 
Observe that, as well known, the problem \eqref{problem-intro} is conformally invariant. This means $\lambda^{\frac {N-2}2}u(\lambda x)$ is also a solution for any $\lambda>0$ if $u$ is a solution. Then it is easy to check that $Z_{\gamma,\lambda}(x)=\frac {\partial u_{\gamma,\lambda}}{\partial \lambda}$ solves \eqref{introeq:linearized2} for every $\gamma$. Therefore,
$$Z_{\gamma,\lambda}(x)=\zeta_{\gamma,\lambda}(x)\left(1-   \lambda^{(2-s)\nu_\gamma}  |x|^{(2-s)\nu_\gamma}  \right),$$ where
$$\zeta_{\gamma,\l}(x)=
  \frac{ \lambda ^{\frac{N-2}2\left(\nu_\gamma-1\right)} |x|^{\frac{N-2}2\left(\nu_\gamma-1\right)}}{\left( 1+ \lambda ^{(2-s)\nu_\gamma} |x|^{(2-s)\nu_\gamma}\right)^{\frac {N-s}{2-s}}}$$
 belongs to the kernel of $L_{\gamma,\lambda}$ for every $\gamma$. For $\gamma>0$, the kernel is always one-dimensional whereas harmonic polynomials generate degeneracies when $\gamma$ takes one of the nonpositive values
\begin{equation}\label{eq:gamma_j}
\gamma_j=\frac{(N-2)^2}4-\frac{j(N-2+j)(N-2)^2}{(2-s)(2N-2-s)},\ j\in\mathbb {N}.
\end{equation}
For $s=0$, $\gamma_1=0$ otherwise all those $\gamma_j$'s are negative. 
The next theorem gives the precise statement. The set 
\begin{equation}\label{eq:harmonics}
\Upsilon_j:=\{Y_{j,i}\mid i=1,\dots,\frac{(N+2j-2)(N+j-3)!}{(N-2)!j!}\}
\end{equation}
denotes a basis of the space of all homogeneous harmonic polynomials of degree $j$ in $\R^N$ (see e.g. \cite{DX}), and we fix the notation
\begin{equation} \label{eq:def-Z-gamma-2}
Z_{j,i,\lambda }(x)= \zeta_{\gamma,\lambda}(x) \lambda ^{\frac{2-s}2 \nu_{\gamma}}|x|^{\frac{2-s}2 \nu_{\gamma}}Y_{j,i}(x).
\end{equation}
\begin{theorem}\label{introlemma-lin}
Assume $\gamma< \frac{(N-2)^2}4$. 
\begin{enumerate}
 \item If $\gamma\neq \gamma_j$, the kernel of $L_{\gamma,\lambda}$  is one-dimensional and it is spanned by the function
  $Z_{\gamma,\lambda}(x)$.
 
 \item If $\gamma=\gamma_j$, the kernel has dimension $1+\frac{(N+2j-2)(N+j-3)!}{(N-2)!\,j!} $ and it is spanned by 
 $\{Z_{\gamma_j,\lambda}, Z_{j,i,\lambda } \mid i=1,\dots,\frac{(N+2j-2)(N+j-3)!}{(N-2)!j!}\}$.
\item the Morse index of $u_{\gamma,\lambda}$  is independent on $\lambda$ and is given by 
$$m(\gamma)
= \sum_{j\in I_\gamma\cap \N} \frac{(N+2j-2)(N+j-3)!}{(N-2)!\,j!},$$
where $I_\gamma:=\left[0,\frac{2-N}2+\frac1 2 \sqrt{(N-s)^2-\frac{4\gamma}{(N-2)^2}(2-s)(2N-2-s)}\right)$.
\end{enumerate}
\end{theorem}




The case $s=0$ is covered by \cite[Lemma 1.2 \& Proposition 1.3]{DGG}. For $s\ne 0$, we improve the result of Robert \cite{Robert} who showed one-dimensionality of the kernel when $\gamma+s>0$. Assertion $(2)$ implies $u_{\gamma,\lambda}$ is nondegenerate when $\gamma\neq \gamma_j$ and in particular for $\gamma \in (\gamma_1 , \frac{(N-2)^2}{4})$. Nondegeneracy here means that the unique element in the kernel comes from the natural invariance of the problem. 
Assertion $(3)$ is the basis to deduce multiplicity results using bifurcation theory. To kill the conformal invariance of the problem, following \cite{MW} (and \cite{GGT2}), we will work in the functional space
\[X:= D^{1,2}_k (\R^N) \cap L^\infty(\R^N),\]
where $ D^{1,2}_k (\R^N)$ is the subset of functions in $ D^{1,2}(\R^N)$ which are invariant by the Kelvin transform, namely
\[D^{1,2}_k:= \ \{z\in D^{1,2}(\R^N)\ : \ z(x)= \frac 1{|x|^{N-2}}z\left(\frac x{|x|^2}\right) \ \text{ in }\R^N\setminus\{0\}\}.\]
Observe that $u_{\gamma ,\lambda}\in X$ if and only if $\lambda=1$. This is a general fact. Indeed, if we define $z_\lambda(\cdot)=\lambda^{(N-2)/2}z(\lambda\cdot)$ for any $z\in D^{1,2}(\R^N)$, then a direct computation shows there exists at most one $\lambda>0$ such that $z_\lambda$ is Kelvin invariant. This means that the Kelvin invariance kills the conformal invariance of the solution as it selects one particular solution in the conformal class. 


\medbreak
%
%

One can also see the Kelvin invariance as a weighted symmetry on the sphere. Assume that $z\in D^{1,2}_k(\R^N)$. 
Set $P=(0,\ldots,0,1)$ and let $\sigma$ be the stereographic projection from $S^{N-1}\setminus\{P\}$ to $\R^N$. On the sphere $S^{N-1}\setminus\{P\}$, we define 
$Z(\zeta,\xi) = z(\frac{\zeta}{1-\xi})$.  Then $z(x)=Z\left(\frac{2x}{1+|x|^2},\frac{|x|^2-1}{|x|^2+1}\right)$ and the Kelvin invariance of $z$ means on the sphere that 
$$(1+\xi)^{\frac{N-2}2}Z(\zeta,\xi) = (1-\xi)^{\frac{N-2}2}Z(\zeta,-\xi).$$ 

\medbreak

The Kelvin invariance in turn has an impact on the nondegeneracy  in $X$ of $u_{\gamma ,1}$ that we simply denote hereafter by $u_\gamma$. Indeed, the function $Z_\gamma:=Z_{\gamma,1}$ is not invariant by Kelvin transform. This fact is at the origin of the following bifurcation result. In the statement, $O(k)$ denotes as usual the orthogonal group in $\R^k$. Before stating the first bifurcation result, observe that taking not only $D^{1,2}_k (\R^N) $ but $D^{1,2}_k (\R^N) \cap L^\infty(\R^N)$ as function space has another important consequence. Indeed, if $z\in X$, then $z(x)|x|^{N-2} = z\left(\frac x{|x|^2}\right)$ is bounded and this directly gives a polynomial decay at infinity. On the contrary, if $z\in D^{1,2}_k (\R^N) $ decays faster than $1/|x|^{N-2}$ at infinity, then it is bounded at the origin. 
\begin{theorem}\label{introteo:bif}
Fix $j\in \N$. 
\begin{enumerate}[(i)]
 \item There exists at least one continuum $\mathcal{C}_j^{N-1}$ of non-radial weak solutions to \eqref{problem-intro}, $O(N-1)$ invariant, which bifurcates from $(\gamma_j,u_{\gamma_j})$ in $(-\infty,0]\times X$.
\item If $j$ is even, there exist at least $\big[\frac N2\big]$ continua $\mathcal{C}_j^\ell$, $\ell\in [1,\ldots ,\big[\frac N2\big] ],$ of non-radial weak solutions to
\eqref{problem-intro} bifurcating from $(\gamma_j,u_{\gamma_j})$ in $(-\infty,0]\times X$. 
The first branch is $O(N-1)$ invariant, the others are $O(N-k)\times O(k)$ invariant for $k=2,\ldots,\big[\frac N2\big]$.

\item  the classical alternative holds: either (a) $\mathcal C_j^\ell$ is unbounded in $(-\infty,0)\times X$ or (b) there exists $ \gamma_h$ with $h\neq j$ such that  $( \gamma_h, u_{\gamma_h})\in\mathcal C_j^\ell$.
\end{enumerate}
\end{theorem}

It is worth mentioning that Jin, Li and Xu \cite{JLX} had previously proved a bifurcation result in the case $s=0$. However their solutions do not belong to $D^{1,2}$ since they arise from the singular solution. Concerning the assertion (iii), one naturally would like to reinforce the alternative (a) by proving the branch is unbounded in the $\gamma$ direction. This requires a priori bounds for the solutions in $X$ which seems hard to get in a bifurcation setting as no energy bounds a priori hold. In the case $s=0$, Musso and Wei \cite{MW} have built an unbounded sequence of solutions so that an a priori bound along a branch, if any, should come from special properties of the solutions along that branch. The case $s\ne 0$ might be more rigid since a solution cannot concentrate outside the origin. One could therefore ask if the mere Kelvin invariance along the branch is enough to guarantee that it is unbounded in the $\gamma$ direction. We leave that as an open question.  In Section \ref{se:6}, we find solutions by minimization in symmetry classes for any given $\gamma$. This supports the conjecture that some branches should be unbounded in the $\gamma$ direction.

To get a better picture, one would like also to rule out the alternative $(b)$ in the assertion $(iii)$ of Theorem \ref{introteo:bif}. To do so, we borrow some ideas from \cite{G}. We introduce spherical coordinates, denoting by $\theta$ the azimuthal angle. We  refer to the beginning of Section \ref{se:4} for the full details. We next define $X^{N-1}$ as the subspace of $X$ given by the functions which are invariant by the action of the orthogonal group $O(N-1)$ in $\R^{N-1}$, which acts on the first $(N-1)$-variables. We then introduce the cones
\begin{equation}\label{def:K1}
\mathcal{K}_{+}^1=\left\{v\in X^{N-1}, 
v(r,\theta) \hbox{ is non-decreasing in } \theta \hbox{ for } (r, \theta)\in (0,\infty)\times[0,\pi] \right\} 
\end{equation}
and
\begin{equation}\label{def:K2}
\mathcal{K}_{+}^2=\left\{\begin{array}{l} v\in X^{N-1},  v(x',x_N)=v(x',-x_N) \hbox{ for }(x',x_N)\in \R^N\setminus\{0\},\\
v(r,\theta) \hbox{ is non-decreasing in } \theta \hbox{ for } (r, \theta)\in (0,\infty)\times[0,\frac \pi 2]
\end{array}\right\}.
\end{equation}
We define similarly $\mathcal{K}_{-}^1$ (resp. $\mathcal{K}_{-}^2$) assuming $v(r,\theta)$ is non-increasing in $(0,\infty)\times[0,\pi]$ (resp. in  $(0,\infty)\times[0,\frac \pi 2]$).
It is natural to search for solutions in these cones since the operator
\begin{equation}\label{def:T}
T(v):=\left( -\Delta-\frac \gamma{|x|^2}I\right)^{-1}\left( 
 \frac{|v|^{p_s-2}v}{|x|^s}\right)\end{equation}
maps $\mathcal K_{\pm}^i$ into itself, for $i=1,2$. Moreover, the change of the Morse index of $u_\gamma$
 in these cones at $\gamma_j$ is odd. This allows us to prove that the Rabinowitz alternative holds.
\begin{theorem}
\label{introteo-bif-Kpm} 
Let $j=1$ or $j=2$. The point $(\gamma_j,
u_{\gamma_j} )$ is a non-radial bifurcation point in $(-\infty,0]\times\mathcal K^j_{\pm}$ and the continuum $\mathcal{C}_j^{\pm}$ that branches out
of $(\gamma_j,u_{\gamma_j})$ is unbounded in $(-\infty,0]\times \mathcal K^j_{\pm}$. Moreover, we have $\mathcal C_1^{\pm}\cap\mathcal C_2^{\pm}=\emptyset$.
\end{theorem}

 Other invariant subspaces can be considered. The symmetries of these subspaces are induced by those of the spherical harmonics $Y_{j,i}$ introduced in \eqref{eq:harmonics}. Let us describe another choice: let $X_j$ be the subset of $X$ invariant by rotations of angle $\frac{2\pi}{j}$ in the plan $(x_1,x_2)$ and by the reflexion with respect to the hyperplane $x_2=0$. Using polar coordinates $(\rho, \psi)$ in the plan $(x_1,x_2)$, we define the cones, for $j\in\mathbb N$,
 \begin{equation}\label{K-j}
\tilde{\mathcal K}^j_{+}\ : \ =
\left\{\begin{array}{l}v\in X_j, 
v(\rho,\psi, x_3,\dots,x_N) \hbox{ is non-decreasing in }\\ \psi \hbox{ for } (\rho, \psi)\in (0,\infty)\times[0,2\pi] , (x_3,\dots,x_N)\in\R^{N-2}
\end{array}\right\}.
\end{equation}
The cones $\tilde{\mathcal K}^j_{-}$ are defined similarly.

\begin{theorem}
\label{introteo-bif-a} 
Let $j\in\N$. The points $(\gamma_j,
u_{\gamma_j} )$ are non-radial bifurcation points in $(-\infty,0]\times \tilde{\mathcal K}^j_{\pm}$. The continuum $\tilde {\mathcal{C}}_j^{\pm}$ that branches out
of $(\gamma_j,u_{\gamma_j})$ is unbounded in $(-\infty,0]\times \tilde{\mathcal K}^j_{\pm}$ and if $j\neq k$, we have $\tilde{\mathcal C}_j ^{\pm}\cap \tilde{\mathcal C}_k^{\pm} \subset \mathcal{X}_\psi$, where $\mathcal X_\psi$ is the subspace of functions in $X$ that do not depend on the angle $\psi$.
\end{theorem}

The plan of the paper is the following : in Section \ref{sectionradial}, we prove the one-to-one correspondance between \eqref{problem-intro} and \eqref{one-to-onecorre} as well as Theorem \ref{introlemma-lin}. Section \ref{se:3} is devoted to the proof of Theorem \ref{introteo:bif}. Theorem \ref{introteo-bif-Kpm} is proved in Sections \ref{se:4} while Theorem \ref{introteo-bif-a} is proved in Section \ref{se:5}. 
Section \ref{se:6} is devoted to the study of existence and symmetry properties of minimisers of the standard functional associated to \eqref{problem-intro}. Finally, in the Appendix \ref{se:7}, we collect several basic results, useful throughout the paper, on linear equations involving the operator $-\Delta - \frac{\gamma}{|x|^2}$ such as a Maximum Principle, a Comparison Principle and some decay estimates.

\subsection{Notations}
 For convenience, we recall here all the notations 
\begin{itemize}
\item $N$ is the dimension and we always assume $N\ge 3$;
\item the ranges for the parameters $\gamma$ and $s$ are given in \eqref{eq:assum-param-s-gamma};
\item $p_s=\frac{2(N-s)}{N-2}$.
\item $\nu_{\gamma}:=\sqrt{1-\frac {4\gamma} {(N-2)^2}}$.
\item $\gamma_j=\frac{(N-2)^2}4-\frac{j(N-2+j)(N-2)^2}{(2-s)(2N-2-s)},\ j\in\mathbb {N}$.
\item  $C_\gamma
=(N-s)(N-2)\nu_\gamma^2$.
\item $a_{\gamma}:=\frac {N-2}2(1-\nu_\gamma)$ and $b_{\gamma,s}:=\frac 2{(2-s)\nu_{\gamma}}$.
\item $X:= D^{1,2}_k (\R^N) \cap L^\infty(\R^N),$ where 
\[D^{1,2}_k:= \ \{z\in D^{1,2}(\R^N)\ : \ z(x)= \frac 1{|x|^{N-2}}z\left(\frac x{|x|^2}\right) \ \text{ in }\R^N\setminus\{0\}\}.\]
\item $O(k)$ is the orthogonal group in $\R^k$.
\end{itemize}

\section{Radial solutions}\label{sectionradial}
In this section, we consider radial solutions to \eqref{problem-intro}. We give a simple proof of their characterization in terms of a nonsingular ODE which allows us to identify the degeneracy and to compute the Morse index. To simplify the expression of the set of solutions, we consider hereafter the problem
\begin{equation}\label{problem}
\begin{cases}
-\Delta u-\frac \gamma{|x|^2}u= C_\gamma
\frac{|u|^{p_s-2}u}{|x|^s} & \text{ in } \R^N\setminus\{0\}\\
u\geq 0 & \text{ in }\R^N\setminus\{0\}\\
u\in D^{1,2}(\R^N)
\end{cases}\end{equation}
with 
$C_{\gamma}=(N-s)(N-2)\nu_\gamma^2$
 where $\nu_\gamma$ as defined in \eqref{nu-gamma}. 
Let us point out that if $U$ is a solution to \eqref{problem}, the function $u(x)=c^{-2}U(cx)$ solves \eqref{problem-intro} provided that 
$c^{2(p_s-1)}C_{\gamma}=1$.
In the sequel, we will prove all the results for solutions to \eqref{problem}.

\

Since we are working in the space $D^{1,2}(\R^N)$, by solution of \eqref{problem} we mean a function $u\in D^{1,2}(\R^N)$ that satisfies 
\[\int_{\R^N}\nabla u\nabla \psi \ dx-\gamma \int_{\R^N}\frac {u \psi}{|x|^2} \ dx= C_\gamma \int_{\R^N}\frac { |u|^{p_s-2}u\psi}{|x|^s} \ dx\]
for any $\psi \in C^{\infty}_0(\R^N)$ and so, by density, for every $\psi \in D^{1,2}(\R^N)$.

We will show that \eqref{problem} restricted to radial solutions yields the one-dimensional problem
\begin{equation}\label{eq:radial-transf}
  \begin{cases}
    -(r^{q_s-1}v')'= q_s(q_s-2)r^{q_s-1}
    v^{\frac{q_s+2}{q_s-2}} & \text{ for } r \ \in (0,\infty)\\
    v\geq 0, \  \ \    \int_0^{\infty}r^{q_s-1}\left(v'(r)\right)^2\ dr <\infty
  \end{cases}
\end{equation}
where $q_s:=\frac {2(N-s)}{2-s}>2$ plays the role of a fractional dimension when $s>0$. Notice that when $s=0$, we recover the classical problem
$$\begin{cases}-\Delta v = N(N-2) v^{\frac{N+2}{N-2}},\\ v\geq 0,\ v\in D^{1,2}_{\rad}(\R^N), \end{cases}$$
where $D^{1,2}_\rad (\R^N)$ is the subspace of radial functions in $D^{1,2}(\R^N)$.
To prove the equivalence between \eqref{problem} and \eqref{eq:radial-transf}, we now introduce some notations. For any $\gamma\in (-\infty, \frac {(N-2)^2}4)$, we define
\begin{equation}\label{a-gamma}
a_{\gamma}:=\frac {N-2}2(1-\nu_\gamma)\  \  \  \ \text{and} \  \  \ \ 
b_{\gamma,s}:=\frac 2{(2-s)\nu_{\gamma}}
\end{equation}
and for any $k>2$, denote by $\mathcal H_k:=D^{1,2}\left((0,\infty),r^{k-1} dr\right)$ the space of measurable functions $w\ : \ (0,\infty)\rightarrow \R$ such that
\begin{equation}\label{eq:norm-k}
  \norm{w}_k=\left(\int_0^\infty r^{k-1} (w')^2\ dr \right)^{\frac 12} <\infty .
\end{equation}
 Observe that $D^{1,2}_\rad (\R^N)=D^{1,2}\left((0,\infty),r^{N-1} dr\right)=\mathcal H_N$,  see \cite{DGG}.

\

Concerning problem \eqref{eq:radial-transf}, we say that a function $v\in \mathcal H_{q_s}$ is a weak solution if it satisfies
\[\int_0^{\infty}r^{q_s-1}v'\xi'\, dr= q_s (q_s -2) \int_0^{\infty}r^{q_s-1}  v^{\frac{q_s+2}{q_s-2}} \xi\, dr,\]
for any $\xi\in C^{\infty}_0[0,\infty)$ and, by density for any $\xi\in \mathcal H_{q_s}$. We can now state the equivalence between the two problems.

\begin{lemma}\label{lem:equiv-rad}
  Let $u\in D_{\rad}^{1,2}(\R^N)$ be a weak solution to \eqref{problem}.
Then the function
\begin{equation}\label{def-v}
v(r):=r^{a_{\gamma}b_{\gamma,s} }u\left( r^{b_{\gamma,s}}\right)
\end{equation}
belongs to $\mathcal H_{q_s}$
and satisfies weakly \eqref{eq:radial-transf}. 
Moreover it holds
\begin{equation}\label{norm-equivalence}
\begin{split}
\int_{\R^N}|\nabla u|^2-\frac {\gamma}{|x|^2}u^2 \, dx& =\omega_N\int_0^{\infty}r^{N-1}\left( (u'(r))^2-\frac {\gamma}{r^2}u(r)^2\right)\, dr \\
&=\omega_N\frac {(2-s)\nu_\gamma}2 \int_0^{\infty}r^{q_s-1}\left( v'(r)\right)^2 \, dr
\end{split}
\end{equation}
where $\omega_N$ is the measure of the unit sphere in $\R^{N}$.
\end{lemma}
The previous lemma also holds true in balls or annuli provided that $u$ and $v$ satisfy either a Dirichlet or a Neumann boundary condition. Notice that, if $v$ satisfies a Dirichlet boundary condition, the condition $\int_0^\infty r^{q_s-1}(v')^2\ ds<\infty$ is equivalent to require $v'(0)=0$ in \eqref{eq:radial-transf} as shown in \cite{AG18}, in a very similar case.

\begin{proof}
First, we show formally that $v$ satisfies \eqref{eq:radial-transf} in a strong sense provided that $u$ is a strong solution to \eqref{problem}. Then, we prove \eqref{norm-equivalence} which will implies that $v\in \mathcal H_{q_s}$ and that $v$ satisfies weakly \eqref{eq:radial-transf}.
To simplify notation, we drop the subscripts if there is no possible confusion. A straight-forward computation gives
\begin{align*}
v^{\prime \prime}(r)+ \dfrac{q-1}{r}v^\prime (r)&= b^2 r^{ab+2(b-1)} \left( u^{\prime \prime}(r^b) +\left( 2a+1 + \frac{q-2}{b}\right) \frac{ u^\prime (r^b) }{r^b} \right.  \\
& \ \ \ \ \ +\left.  (a^2 - \frac{a}{ b}+\frac{(q-1) a}{b} )\frac{ u (r^b) }{r^{2b}} \right) .
\end{align*}
Using the definitions of $a,b$ and $q$, one can check that
\begin{equation}
\label{lem:equiv-rade1}
2a+1 + \frac{q-2}{b} =N-1,
\end{equation}
and
$$a^2 - \frac{a}{ b}+\frac{(q-1) a}{b}  = (\dfrac{N-2}{2})^2  (1- \nu^2) = \gamma .$$
Thus, noticing that $ab + 2(b-1)-bs - (p-1) ab=0 $ and using \eqref{problem}, we find
\begin{align*}
v'' + \frac{q-1}{r} v' &=C_\gamma b^2 r^{ab+2(b-1)-bs } |u|^{p-2} u\\
&= C_\gamma b^2 r^{ab + 2(b-1)-bs - (p-1) ab} |v|^{p-2} v\\
&= q(q-2)  |v|^{p-2} v.
\end{align*}
This establishes that $v$ is a solution to \eqref{eq:radial-transf}. Next, we are going to prove \eqref{norm-equivalence}. Using a change of variables and \eqref{lem:equiv-rade1}, we obtain

\begin{align*}
\int_0^\infty r^{q-1} (v^\prime (r))^2 dr &= \int_0^\infty r^{q-1} (ab r^{ab-1} u(r^b)+b r^{ab+b-1} u^\prime (r^b))^2 dr \\
&= b \int_0^\infty t^{ 2a+1 + \frac{q-2}{b}} (u^\prime (t) +a \frac{u(t)}{ t} )^2 dt\\
&= b \int_0^\infty \left( t^{N-1}  (u^\prime (t))^2 + a (t^{N-2} u^2 (t) )^\prime + (a^2- a (N-2))   t^{N-3}  u^2 (t) \right) dt\\
&= b \int_0^\infty t^{N-1} ( (u^\prime (t))^2 - \gamma \dfrac{u^2 (t)}{t^2}  )dt.
\end{align*}
This concludes the proof.
\end{proof}

\

As a corollary of the previous result we have the following

\begin{corollary}\label{cor-def-U}
  For  any $\gamma\in (-\infty, \frac {(N-2)^2}4)$ and $s\in[0,2)$, problem \eqref{problem} admits a unique (up to the dilation 
 $\lambda ^{\frac {N-2}2} U(\lambda x)$, $\lambda >0$) radial solution which is given by
  \begin{equation}\label{def:u-gamma-s}
     U_{\gamma}(x):=\frac{ |x|^{\frac {N-2}{2}(\nu_\gamma-1)}}{\left(1+|x|^{(2-s)\nu_\gamma}\right)^{\frac {N-2}{2-s}}}.
       \end{equation}
    \end{corollary}

\begin{proof}
This is a direct consequence of Lemma \ref{lem:equiv-rad} and the facts that problem \eqref{eq:radial-transf} admits only the solution $V(r)=\left(\frac 1{(1+r^2)}\right)^{\frac {q_s-2}2}$, up to the scaling $S_\lambda  (V)=\lambda ^{\frac{q_s-2}2}V(\lambda x)$ with $\lambda>0$
and the solution $W(r)=\left(\frac 1 r\right)^{\frac{q_s-2}2}$. This second solution does not belong to $\mathcal H_{q_s}$, see \cite{Talenti} or \cite{Caffarelli-Gidas-Spruck}.
\end{proof}
  
\

Moreover it has been proved that 
\begin{proposition}\label{prop:existence}
Problem \eqref{problem} does not admit solutions for $\gamma\geq \frac{(N-2)^2}4$ and admits only radial solutions for $\gamma\in (\gamma_1,  \frac{(N-2)^2}4)$.
\end{proposition}
The nonexistence part when $\gamma \geq \frac{(N-2)^2}4$ has been obtained in \cite{Smets}[Proposition 2.1]. Concerning the radial symmetry for $\gamma\in (\gamma_1,  \frac{(N-2)^2}4)$, we first point out that $\gamma_1\leq 0$ with equality if and only if $s=0$. The case $s=0$ has been established in \cite{T} using the moving planes method. See also \cite{Chou-Chu} for the case of $0<s<2$ and $\gamma\geq0 $.
The remaining cases follow from \cite{DEL-16}, see also \cite{DEL-17}. 

\

\subsection{Degeneracy of radial solutions}
In this section we consider the linearization of equation \eqref{problem} at the radial solution $U_{\gamma}$ given by \eqref{def:u-gamma-s} that is to say
\begin{equation}\label{eq:linearized2}
\begin{cases}
  -\Delta v-\frac \gamma{|x|^2}v=(N-s)(N+2-2s)\nu_\gamma^2\frac{|x|^{\nu_\gamma(2-s)-2}}{\left(1+|x|^{(2-s)\nu_\gamma}\right)^2}
v, & \text{ in } \R^N\setminus\{0\},\\
v\in D^{1,2}(\R^N).
\end{cases}\end{equation}
 Observe that \eqref{eq:linearized2} is exactly the linearization of \eqref{problem-intro} at the solution $u_\gamma$ in \eqref{sol-rad} with $\lambda=1$. In our next result, we classify solutions to \eqref{eq:linearized2}. Before proceeding, we introduce some notation and recall some well-known facts. We denote by $Y_j(\theta)$ the $j$-th spherical harmonics i.e. $Y_j$ satisfies
\begin{equation*}
-\Delta_{S^{N-1}} Y_j = \mu_jY_j ,
\end{equation*}
where $\Delta_{S^{N-1}}$ is the Laplace-Beltrami operator on $S^{N-1}$ with the standard metric and $\mu_j$ is the $j$-th eigenvalue
of $-\Delta_{S^{N-1}}$. It is known that, for any $j\in \mathbb{N}$,
\begin{equation}\label{mu-j}
\mu_j = j(N-2+j), 
\end{equation}
whose multiplicity is given by
\begin{equation}\label{dim-mu-j}
 \frac{(N+2j-2)(N+j-3)!}{(N-2)!\,j!} ,
\end{equation}
and that
$$ \mathrm{Ker}\left(\Delta_{S^{N-1}} + \mu_j\right) = \mathbb{Y}_j(\R^N)\left|_{S^{N-1}},\right.$$
where $\mathbb{Y}_j(\R^N)$ is the space of all homogeneous harmonic polynomials of degree $j$ in $\R^N$.

\begin{proposition}\label{lemma-lin}
Let $\gamma< \frac{(N-2)^2}4$. Denote by $\gamma_j$ the values 
\begin{equation}\label{eq:gamma-j}
\gamma_j=\frac{(N-2)^2}4-\frac{j(N-2+j)(N-2)^2}{(2-s)(2N-2-s)},\ j\in\mathbb {N}.
\end{equation}
If $\gamma\neq \gamma_j$ then the space of solutions of \eqref{eq:linearized2}  has dimension 1 and it is spanned by
\begin{equation}\label{eq:def-Z-gamma}
  Z_\gamma(x)=
  \frac{ |x|^{\frac{N-2}2\left(\nu_\gamma-1\right)}\left(1-|x|^{(2-s)\nu_\gamma}  \right)}{\left( 1+ |x|^{(2-s)\nu_\gamma}\right)^{\frac {N-s}{2-s}}},
\end{equation}
where $\nu_\gamma$ is as defined in \eqref{nu-gamma}.\\
If $\gamma=\gamma_j$, then the space of solutions of \eqref{eq:linearized2} has dimension $1+\frac{(N+2j-2)(N+j-3)!}{(N-2)!\,j!} $ and it is spanned by 
\begin{equation} \label{eq:def-Z-gamma-2}
Z_{\gamma_j}(x) \,\, , \,\, Z_{j,i}(x)=\frac{ |x|^{\frac{N-2}2\left(\nu_\gamma-1\right)+\frac{2-s}2 \nu_{\gamma}}}{\left( 1+ |x|^{(2-s)\nu_\gamma}\right)^{\frac {N-s}{2-s}}} Y_{i}(x),
\end{equation}
where $\{Y_{i}\}$, ${i=1,\dots,\frac{(N+2j-2)(N+j-3)!}{(N-2)!j!}}$, form a basis of $\mathbb{Y}_j(\R^N)$. \end{proposition}
 Theorem \ref{introlemma-lin} $(1)$ and $(2)$ follow by Proposition \ref{lemma-lin} and the scaling invariance of \eqref{problem-intro}.
\begin{proof}
We solve \eqref{eq:linearized2} using the decomposition along the spherical harmonic functions $Y_j(\theta)$. More precisely, we write
\begin{equation*}
v(r,\theta) = \sum_{j=0}^{\infty} \psi_j(r)Y_j(\theta), \qquad \text{where} \qquad r=|x| \,\, , \,\, \theta=\frac{x}{|x|} \in S^{N-1}
\end{equation*}
and
\begin{equation*}
\psi_j(r) = \int_{S^{N-1}}V(r,\theta)Y_j(\theta)\,d\theta.
\end{equation*}
The function $v$ is a weak solution to \eqref{eq:linearized2} if and only if, for any $j\in \N$, $\psi_j(r)$ is a weak solution to
\begin{equation} \label{1.8}
\begin{cases}
-\left(r^{N-1}\psi_j'(r)\right)'+ (\mu_j-\gamma) r^{N-3} \psi_j(r) =\frac{ C_\gamma
(p_s-1) r^{N+(2-s)\nu_\gamma-3}} {\left(1+ r^{(2-s)\nu_\gamma}
\right)^ 2} \psi_j
&
\text{in} \,\,(0,\infty) \\
\psi_j \in \mathcal H_N.
\end{cases}
\end{equation}
Performing the same transformation as in  \eqref{def-v}, we define 
\begin{equation}\label{eq:def-hat-psi}
  \hat \psi_j(r)=r^{a_\gamma b_{\gamma,s}} \psi_j\left(  r^{b_{\gamma,s}}\right) \in \mathcal H_{q_s}.
  \end{equation}
Proceeding as in the proof of Lemma \ref{lem:equiv-rad}, we see that 
\begin{align*}
- \hat{\psi}_j^{\prime \prime}(r) - \frac{q_s-1}{r} \hat{\psi}_j^\prime (r) &= b_{\gamma ,s}^2 r^{ a_\gamma b_{\gamma ,s}+2b_{\gamma ,s}-2 } \left(- \mu_j r^{-2b_{\gamma ,s}} \psi_j (r^{ b_{\gamma ,s}})\right. \\
& \left.+ \,C_\gamma (p_s-1) r^{b_{\gamma ,s} ((2-s)\nu_\gamma -1)} (1+ r^{b_{\gamma ,s} (2-s)\nu_\gamma } )^{-2}\psi_j (r^{ b_{\gamma ,s}}) \right) \\
& = - \mu_j b_{\gamma ,s}^2 r^{-2} \hat{\psi}_j (r)\\
&+ C_\gamma (p_s-1)b_{\gamma ,s}^2 r^{2b_{\gamma ,s} -2 +b_{\gamma ,s} ((2-s)\nu_\gamma -2)} (1+r^2)^{-2} \hat{\psi}_j (r).
\end{align*}

So one can check that $ \hat \psi_j$ solves weakly

\begin{equation} \label{3.5}
\begin{cases}
-\left(r^{q_s-1}\hat\psi_j'(r)\right)'+ \frac{4\mu_j r^{q_s-3}}{(2-s)^2\nu_\gamma^2}\hat \psi_j(r) =\frac{q_s(q_s+2) r^{q_s-1}} {\left(1+ r^2\right)^ 2} \hat \psi_j ,&
\text{in} \,\,(0,\infty) \\
\hat \psi_j \in \mathcal H_{q_s}.
\end{cases}
\end{equation}
The latter can be interpreted as the weighted eigenvalue problem 
\begin{equation}\label{2.17-bis}
-\left(r^{q_s-1}\psi'(r)\right)'-q_s(q_s+2)
\frac{ r^{q_s-1}} {\left(1+ r^2\right)^ 2} \psi=\mu^*r^{q_s-3}\psi , \ \ \text{ in }(0,\infty)
\end{equation}
with $\mu^*=-\frac{4\mu_j }{(2-s)^2\nu_\gamma^2}\leq 0$. This eigenvalue problem is related to the linearization to \eqref{eq:radial-transf} at the solution $V(r)=\left(\frac 1{(1+r^2)}\right)^{\frac {q_s-2}2}$. Since \eqref{eq:radial-transf} is variational, $V(r)$ is a least energy solution and therefore the linearized operator admits at most one negative eigenvalue. The same is true for the  weighted eigenvalue problem \eqref{2.17-bis} that can have at most one negative eigenvalue. This equivalence has been proved in details in \cite[Proposition 3.11]{AG18} in the case of a bounded interval, see also \cite{DGG}. A straightforward computation then shows that 
\begin{itemize}
\item[$i)$] $\mu^*=-(q_s-1)$
is the unique negative eigenvalue to \eqref{2.17-bis} and it is related to the eigenfunction $$\hat \psi(r)=\frac r{(1+r^2)^\frac {q_s}2}.$$
\item [$ii)$] $\mu^*=0$
is a non positive eigenvalue to \eqref{2.17-bis} and it is related to the function $$\hat \psi_0(r)= \frac{1-r^2}{(1+r^2)^{\frac {q_s}2}}.$$ 
\end{itemize}
See \cite{Ambrosetti-Azorero-Peral}. Since $\mu^*=0$ implies $\mu_j=0$ so that $j=0$, scaling back $\hat \psi_0$ and recalling that $Y_0$ is constant, we find that
$$\psi_0(r)=r^{-a_\gamma}\hat \psi_0(r^\frac 1{b_\gamma})= \frac{ r^{\frac{N-2}2\left(\nu_\gamma-1\right)}\left(1-r^{(2-s)\nu_\gamma}  \right)}{\left( 1+ r^{(2-s)\nu_\gamma}\right)^{\frac {N-s}{2-s}} }$$
is a solution to  \eqref{eq:linearized2}, for any value of $\gamma$.

The case $\mu^*=-(q_s-1)$ instead implies that
\begin{equation}\nonumber
\gamma_j=\frac{(N-2)^2}4\left(1-\frac {4\mu_j}{(2-s)^2 (q_s-1)}\right)
\end{equation}
and 
rescaling back, we find that  
$$\psi_j(r)=r^{-a_\gamma}\hat \psi(r^\frac 1{b_\gamma})=\frac{ r^{\frac{N-2}2\left(\nu_\gamma-1\right)+\frac{2-s}2 \nu_{\gamma}}}{\left( 1+ r^{(2-s)\nu_\gamma}\right)^{\frac {N-s}{2-s}}}$$
is a solution to \eqref{1.8} if and only if $\gamma=\gamma_j$.

Thus, when $\gamma=\gamma_j$, the solutions to \eqref{eq:linearized2} are given by the function $Z_\gamma$ in \eqref{eq:def-Z-gamma} (corresponding to $\gamma=\gamma_j$) and by the functions $\psi_j(|x|)Y_{j,k}(\theta)$ where 
$k=1,\dots,   \frac{(N+2j-2)(N+j-3)!}{(N-2)!\,j!}$.
This proves \eqref{eq:def-Z-gamma-2} and finishes the proof.
\end{proof}

\

\begin{remark}By Proposition \ref{lemma-lin} we obtain a sequence of degeneracy points $\gamma_j$, $j\in \N$, such that the radial solution $U_\gamma$
 is degenerate at $\gamma=\gamma_j$. These degeneracy points are isolated and accumulate at $-\infty$. For $j=0$, we obtain the value $\gamma= \frac {(N-2)^2}4$ which is the threshold between existence and nonexistence of solutions to \eqref{problem} (see Proposition \ref{prop:existence}). For $j=1$, we obtain the first degeneracy point $\gamma_1=(N-2)^2\left( \frac 14-\frac{(N-1)}{(2-s)(2N-2-s)}\right)$ of the curve  $U_\gamma$ which is equal to $0$ when $s=0$ and strictly negative when $s>0$. Moreover $\gamma_1\to -\infty$ as $s\to 2^-$.  
    \end{remark}
    
    \
    
    As a corollary to Proposition \ref{lemma-lin}, we get the following
    \begin{corollary}\label{cor:nondegeneracy}
    The radial solution $U_\gamma$ is nondegenerate for $\gamma\in (\gamma_1,  \frac {(N-2)^2}4)$.
    \end{corollary}
    Here by nondegenerate we mean that it admits only the degeneracy due to the dilation invariance of the problem, which is given by the function $Z_\gamma$ in \eqref{eq:def-Z-gamma}. Corollary \ref{cor:nondegeneracy} simplifies a previous result of Robert in \cite{Robert} and extends it until the first negative value $\gamma_1$ when $s>0$. Moreover Proposition \ref{lemma-lin} proves that the nondegeneracy holds true except for the sequence of values $\gamma_j$ in \eqref{eq:gamma-j}.
    
    \

\subsection{Morse index of radial solutions}
In this section we compute the Morse index of the radial solutions $U_\gamma$
 depending on the parameter $\gamma$. We first recall that the Morse index of a radial solution $U_\gamma$ to \eqref{problem} is the maximal dimension of a subspace of $D^{1,2}(\R^N)$ such that the quadratic form corresponding to the linearized operator, namely
\[Q(\psi,\psi):=\int_{\R^N}|\nabla\psi|^2-\gamma \frac{\psi^2}{|x|^2}-(N-s)(N+2-2s)\nu_\gamma^2\frac{|x|^{\nu_\gamma(2-s)-2}}{\left(1+|x|^{(2-s)\nu_\gamma}\right)^2}\psi^2 \ dx
\]
is negative definite. Since the linearized operator is compact in $D^{1,2}(\R^N)$, it admits a sequence of eigenvalues $\l_1<\l_2<\dots$ such that $\l_n\to +\infty$ as $n\to \infty$ with eigenfunctions in $D^{1,2}(\R^N)$
and the Morse index of the radial solution
$U_\gamma$ is finite and coincides with the number, counted with multiplicity, of negative eigenvalues of the linearized operator 
\begin{equation}\label{eq:autov-lin}
\begin{cases}
 L_\gamma v:= -\Delta v-\frac \gamma{|x|^2}v-(N-s)(N+2-2s)\nu_\gamma^2\frac{|x|^{\nu_\gamma(2-s)-2}}{\left(1+|x|^{(2-s)\nu_\gamma}\right)^2}
v=\l v, & \text{ in } \R^N\setminus\{0\},\\
v\in D^{1,2}(\R^N).
\end{cases}\end{equation}
To simplify the computation of the Morse index, instead of considering the eigenvalue problem \eqref{eq:autov-lin}, we consider an auxiliary eigenvalue problem associated to the same linearized operator,
\begin{equation}\label{eq:autov-lin-sing}
\begin{cases}
  L_\gamma v:= -\Delta v-\frac \gamma{|x|^2}v-(N-s)(N+2-2s)\nu_\gamma^2\frac{|x|^{\nu_\gamma(2-s)-2}}{\left(1+|x|^{(2-s)\nu_\gamma}\right)^2}
v=\Lambda \frac {v}{|x|^2}, & \text{ in } \R^N\setminus\{0\},\\
v\in D^{1,2}(\R^N).
\end{cases}\end{equation}
These eigenvalues are well defined thanks to Hardy inequality. Moreover in \cite[Proposition 3.1]{AG18} (see also \cite{DGG} for previous results), it is proved that they are attained when $\gamma+\Lambda <\frac{(N-2)^2}4$. Since the Morse index of $U_\gamma$ only involves the negative eigenvalue $\L_i$, then $\gamma+\Lambda<\frac{(N-2)^2}4$, for every $\gamma<\frac{(N-2)^2}4$ and all these eigenvalues are attained.
 Moreover the following correspondence with the classical eigenvalues $\l_i$ holds, see \cite{DGG} and \cite[Proposition 1.1]{AG18}. 
\begin{lemma}
The number of negative eigenvalues $\Lambda_i$ of \eqref{eq:autov-lin-sing}, counted with multiplicity, coincides with the number of negative eigenvalues 
$\l_j$ of \eqref{eq:autov-lin} counted with multiplicity.\end{lemma}
As a corollary of the previous result, we obtain

\begin{corollary}\label{cor-morse-index}
The Morse index of the radial solution $U_\gamma$
 is given by the number of negative eigenvalues $\Lambda_i$ of the auxiliary problem \eqref{eq:autov-lin-sing}, counted with multiplicity.
\end{corollary}

\
We point out that the previous corollary also holds if we work in $ D_{\mathcal G}^{1,2}(\R^N)$, where  $\mathcal G$ is a group of  transformation from $\R^N\setminus\{0\}$ into itself, provided that $L_\gamma$ is invariant under $\mathcal G$, see \cite{AG18}. Above, $D^{1,2}_{\mathcal G}(\R^N)$ denotes functions in $D^{1,2}(\R^N)$ which are invariant by the action of $\mathcal G$. We denote by $m^{\mathcal G}(\gamma,s)$ the corresponding Morse index. As a special case, we can take $\mathcal G=\mathcal{O} (N)$, the orthogonal group, then $D^{1,2}_{\mathcal G} (\R^N)= D_{\rad}^{1,2}(\R^N)$ and we denote the radial Morse index by $m^{\rad}(\gamma)$. 
Next, we show that problem \eqref{eq:autov-lin-sing} admits a unique explicit radial eigenfunction.
\begin{lemma}\label{lem:first-radial-sing-eigen}
For every $\gamma\in (-\infty,\frac {(N-2)^2}4)$, there is a unique negative radial singular eigenvalue $\Lambda_1^\rad$ of problem \eqref{eq:autov-lin-sing} given by 
\begin{equation}\label{eq:lambda-1-rad}
  \Lambda_1^\rad =-\frac {(2-s)^2\nu_\gamma^2(q_s-1)}4
\end{equation}
whose corresponding eigenfunction is  
\begin{equation}\label{eq:psi-1-rad}
\psi_1^{\rad}(x)=\frac {|x|^{\frac {N-2}2(\nu_\gamma-1)+\frac {2-s}2\nu_\gamma}}{(1+|x|^{(2-s)\nu_\gamma})^\frac{N-s}{2-s}}.\end{equation}
\end{lemma}
\begin{proof}
First, we recall that the radial solution $U_\gamma$ can be obtained by minimizing the functional
$F: D^{1,2}_\rad(\R^N)\to \R$ 
      \begin{equation}\label{eq:functional-F-rad}
        F(u):=\frac 12 \int_{\R^N}|\nabla u|^2-\frac \gamma 2\int_{\R^N}\frac{ u^2}{|x|^2}-\frac {C_\gamma}{p_s} \int_{\R^N}\frac{ |u|^{p_s}}{|x|^s},
          \end{equation}
         on the Nehari set
           \[\mathcal N_\rad:=
            \{u\in D^{1,2}_\rad(R^N) \ u\neq 0:  \int_{\R^N}|\nabla u|^2-\gamma\int_{\R^N}\frac{ u^2}{|x|^2}- C_\gamma \int_{\R^N}\frac{ |u|^{p_s}}{|x|^s}  =0\}. \]
 One can refer to Section \ref{se:6} where this minimization procedure is given in details in the space $D^{1,2}(\R^N)$.  Since $U_\gamma$
  is a minimum on a manifold $\mathcal N_\rad$ of codimension $1$, the radial Morse index of $U_\gamma$ is $1$. So, by the previous corollary, the eigenvalue problem \eqref{eq:autov-lin-sing} admits only one negative radial eigenvalue $\Lambda_1^{\rad}=\Lambda_1$ with corresponding eigenfunction $\psi_1^{\rad}>0$ in $ D^{1,2}_\rad(\R^N):=\mathcal H_N$. 
This function solves weakly \[-(r^{N-1}(\psi_1^\rad)')'-(N-s)(N+2-2s)\nu_\gamma^2\frac{r^{\nu_\gamma(2-s)+N-3}}{\left(1+r^{(2-s)\nu_\gamma}\right)^2}
\psi_1^\rad=(\Lambda_1^\rad +\gamma)r^{N-3} {\psi_1^\rad}, \]
in $(0,\infty)$
with $\Lambda_1^\rad +\gamma<\gamma<\frac{(N-2)^2}4$. Proceeding as in Proposition \ref{lemma-lin}, we find that
$\hat \psi(r)=r^{a_\gamma b_{\gamma,s}}\psi_1^\rad\left(r^{b_{\gamma,s}}\right) \in \mathcal H_{q_s}$ solves  weakly
\[-(r^{q_s-1}\hat \psi')'-q_s(q_s+2)\frac {r^{q_s-1}}{(1+r^2)^2}\hat \psi=\frac {4\Lambda_1^\rad}{(2-s)^2\nu_\gamma^2}r^{q_s-3}\hat \psi ,\ \ \text{ in }(0,\infty).\]
It is well known, see e.g. \cite{Ambrosetti-Azorero-Peral}, that the unique negative eigenvalue of this problem is given by $-(q_s-1)$ with corresponding eigenfunction $\hat \psi(r)=\frac r{(1+r^2)^{\frac {q_s}2}}$. Scaling back, we obtain \eqref{eq:lambda-1-rad} and \eqref{eq:psi-1-rad}.
\end{proof}

\

We are now in a position to compute the Morse index of the radial solution $U_\gamma$ leading therefore to the assertion $(3)$ of Theorem \ref{introlemma-lin}. 
\begin{proposition}\label{prop:Morse-index}
Let $U_\gamma$ be a radial solution to \eqref{problem}. Then its Morse index $m(\gamma)$ is equal to
\begin{equation}\label{eq:morse-index}
m(\gamma) = \sum_{0 \leq  j<\frac{2-N}2+\frac1 2 \sqrt{(N-s)^2-\frac{4\gamma}{(N-2)^2}(2-s)(2N-2-s)}   \atop_{j\  integer} } \frac{(N+2j-2)(N+j-3)!}{(N-2)!\,j!}.
\end{equation}
In particular, the Morse index of $U_\gamma$ changes as $\gamma$ crosses the values $\gamma_j$, defined in \eqref{eq:gamma-j} and  $m(\gamma) \to+\infty$ as $\gamma \to-\infty$ for every fixed $s\in[0,2)$.
\end{proposition}
\begin{proof}
We know, by \cite{DGG} and \cite{AG18},  that the negative singular eigenvalues $\L_i$ of \eqref{eq:autov-lin-sing} can be decomposed in radial and angular part, and the following identity holds
\begin{equation}\label{eq:decomp-autov}
\L_i=\L_1^{\rad}+\mu_j , \end{equation}
for every $i=1,\dots, m(\gamma)$ and for some $j\in \N_0$, where $\L_1^{\rad}$ and $\mu_j$ are defined respectively in Lemma \ref{lem:first-radial-sing-eigen} and \eqref{mu-j}. Moreover the eigenfunctions corresponding to $\L_i$ are given by 
\begin{equation}\label{decomposition-eigenfunctions}
  \psi_i(x)=\psi_1^\rad (|x|)Y_j(\theta)\end{equation}
  and have the same multiplicity of the spherical harmonics corresponding to $\mu_j$, namely $\frac{(N+2j-2)(N+j-3)!}{(N-2)!\,j!}$.
Thus, to compute the Morse index of $U_\gamma$, we only have to add the multiplicity of the spherical harmonics such that $\L_1^\rad+\mu_j<0$. This last inequality can be rewritten as $0\leq j<\frac{2-N}2+\frac 12 \sqrt{(N-2)^2-4\L_1^\rad}$. 
Using the value of $\L_1^\rad$ given in \eqref{eq:lambda-1-rad}, we get \eqref{eq:morse-index}.
\end{proof}

\

As a corollary we have 
\begin{corollary}\label{cor:G-morse-index-2}
Let $U_\gamma$ be a radial solution to \eqref{problem} and $\mathcal G$ be a group of transformations of $\R^N \backslash \{0\}$. Assume that the operator $L_{\gamma}$ is invariant under $\mathcal{G}$. Then the Morse index of $U_\gamma$ in the symmetric space $D^{1,2}_{\mathcal G}$, is equal to
\[m^{\mathcal G}(\gamma) = \sum_{0 \leq  j<\frac{2-N}2+\frac1 2 \sqrt{(N-s)^2-\frac{4\gamma}{(N-2)^2}(2-s)(2N-2-s)}   \atop_{j\  integer} }  \mathrm{dim}Y_j^{\mathcal G}, \]
where we denote by $Y_j^{\mathcal G}$ the spherical harmonics which are invariant by the action of $\mathcal G$.
\end{corollary}

\

\section{A first bifurcation result}\label{se:3}
In this section we turn to the bifurcation problem for which we use a Leray Schauder degree approach which will be extent to an index argument in Section \ref{se:4}.
To this end, we first introduce the functional setting. 
For any function $z$ in $ D^{1,2}(\R^N)$, we denote by $k(z)$ the 
{\em Kelvin transform} of $z$, namely
\[k(z)(x):= \frac 1{|x|^{N-2}}z\left(\frac x{|x|^2}\right), \ \ \text{ in }\R^N\setminus\{0\}\]
which maps $ D^{1,2}(\R^N)$ into itself. Using that $\Delta (k(z)(u))(x)=\frac{1}{|x|^{N+2}}\Delta u (\frac{x}{|x|^2})$, it is easy to check that problem \eqref{problem} is also invariant by Kelvin transform. In particular we denote by $ D^{1,2}_k (\R^N)$ the subset of functions in $ D^{1,2}(\R^N)$ which are invariant by the Kelvin transform, namely
\[D^{1,2}_k:= \ \{z\in D^{1,2}(\R^N)\ : \ z(x)=k(z)(x) \ \text{ in }\R^N\setminus\{0\}\}\]
and we set 
\[
X:= D^{1,2}_k (\R^N) \cap L^\infty(\R^N).\]
The space $X$ is a Banach space with the norm
\[
\norm{g}_X:=\max\{\norm{g}_{1,2},\norm{g}_{\infty}\},  \]
where $\norm{g}_{1,2}^2:=\int_{\R^N}|\nabla g|^2\ dx$ is the standard norm in $D^{1,2}(\R^N)$ and $\norm{g}_{\infty}$ is the standard norm in $L^{\infty}(\R^N)$. 
Observe that any function $v\in X$ satisfies
\begin{equation}\label{eq:decay-v}
  |v(x)|\leq C_v(1+|x|)^{2-N},\end{equation}
  for some constant $C_v$ which depends on $v$. 
We next define the operator $T: \ X\to  X$ as
\begin{equation}\label{def:T}
T(\gamma, v):=\left( -\Delta-\frac \gamma{|x|^2}I\right)^{-1}\left( C_\gamma
 \frac{|v|^{p_s-2}v}{|x|^s}\right),\end{equation}
 so that fixed points for $T$ are solutions to \eqref{problem}. First we show:

\

\begin{lemma}\label{lem:T-well-def}
  The operator $T$ is well defined from $(-\infty,0]\times \ X$ into $ X$.
\end{lemma}
\begin{proof}
  Let $v\in X$. Using \eqref{eq:decay-v},  we have $\frac{|v|^{p_s-2}v}{|x|^s}\sim |x|^{-s}$ at zero and $\frac{|v|^{p_s-2}v}{|x|^s}\sim |x|^{s-N-2}$ at infinity, showing that  
  $\frac{|v|^{p_s-2}v}{|x|^s}\in L^q(\R^N)$ for any $1<q<\frac N s$ (for every $1<q$ if $s=0$). In particular, we have $ \frac{|v|^{p_s-2}v}{|x|^s}\in L^{\frac{2N}{N+2}}(\R^N)$.
Then, by Lemma \ref{lem:appendix-1} 
there exists a unique $g=T(\gamma,v)\in D^{1,2}(\R^N)$ such that $g$ is a weak solution to   
  \begin{equation}\label{eq:linear}
      -\Delta g-\frac \gamma{|x|^2}g= C_\gamma
      \frac {|v|^{p_s-2}v}{|x|^s}\  \ \text{ in }\R^N\setminus \{0\}.
  \end{equation}
Since $v\in D^{1,2}_k (\R^N)$, it is easy to check that the Kelvin transform of $g$, namely $k(g)$ also solves \eqref{eq:linear}. So  $k(g)-g$ is a solution to $-\Delta (k(g)-g)-\frac \gamma{|x|^2} (k(g)-g)=0$. The maximum principle (see Lemma \ref{lem:appendix-2}) then implies that $k(g)=g$ showing that $g$ is Kelvin invariant.\\
Next we show that $g$ belongs to  $L^{\infty}(\R^N)$. 
Since $g$ weakly solves \eqref{eq:linear} and $v$ satisfies \eqref{eq:decay-v}, the comparison principle (Lemma \ref{lem:appendix-3}) implies that 
$|g|\leq C w$, where $w\in D^{1,2}(\R^N)$ is a weak solution to
 \[-\Delta w- \frac \gamma{|x|^2} w=\frac 1{|x|^s (1+|x|)^{(N-2)(p_s-1)}} \ \  \text{in } \R^N\setminus\{0\}.\]
Then, by Lemma \ref{lem-decad-w}, for any $\gamma\leq0$, $g$  belongs to $L^{\infty}(\R^N)$. 
\end{proof}

\

When $0<\gamma<\frac{(N-2)^2}4$ the space $X$ is not the correct space to consider. This is due to the fact that the radial solution $U_\gamma$
 is not bounded at the origin for $\gamma>0$ and therefore does not belong to $L^{\infty}$. This suggests that a singularity of order $\frac{N-2}2(1-\nu_\gamma)$ in the origin has to be allowed to consider this case. But, since we already know that problem \eqref{problem} admits only radial solutions when $\gamma\geq 0$, a non-radial bifurcation can happen only when $\gamma<0$. Thus, it is not restrictive to consider only solutions which belong to $L^{\infty}$.

\

\begin{remark}[Regularity of $g$]\label{rem:regularity}
  In the proof of the previous Lemma, we have shown that if $v\in X$ and $g=T(\gamma,v)$ with $\gamma\leq 0$, then $|g|\leq Cw$, with $w$ defined as in Lemma \ref{lem-decad-w}. This implies that $g(x) \sim |x|^\beta$ at zero and $g(x) |x|^{N-2}\leq C$ at infinity where $\beta:=\min\{2-s, -a_\gamma\}>0$ so that $\frac \gamma{|x|^2}g(x)$ and $\frac{|v|^{p_s-2}v}{|x|^s}$ are in $L^{q}(\R^N)$ for some $\frac N2 <q < \min \{\frac{N}{2-\beta} ,\frac{N}{s}\}$. So, by elliptic regularity theory, we have that $g\in W^{2,q}(\R^N)$ and, since $q>\frac N2$, we also deduce that $g\in C^{0,\alpha}_{loc}(\R^N\setminus\{0\})$ for some $0<\alpha<1$. Furthermore, $\frac \gamma{|x|^2}g(x)$ and $\frac{|v|^{p_s-2}v}{|x|^s}$ belong to $L^{\infty}(\Omega) \cap H^1(\Omega)$ if $\Omega$ is any subset of $\R^N$ such that $\bar \Omega\subset\R^N\setminus\{0\}$, meaning that $-\Delta g=f(x)$ with $f\in L^{q}(\Omega)\cap H^1(\Omega)$ for $q>N$, so that $g\in W^{2,q}(\Omega)\cap H^{3}(\Omega)$ and hence $g\in C^{1,\a}_{loc}(\R^N\setminus\{0\})$.

  \end{remark}

Our choice of $X$ or more precisely its invariance by Kelvin transform is motivated by the fact that it ``cancels'' the dilatation invariance of problem \eqref{problem}. The following simple uniqueness result concerning radial solutions is an illustration of this fact.
\begin{lemma}
\label{uniqreadialX}
For any $\gamma\leq 0$ fixed, the operator $T$ in \eqref{def:T} admits a unique radial fixed point in $X$, given by 
\[ U_\gamma (x):=\frac{|x|^{\frac {N-2}2(\nu_\gamma-1)}}{\left(1+|x|^{(2-s)\nu_\gamma}\right)^{\frac {N-2}{2-s}}}.
       \]
\end{lemma}
\begin{proof}
By Corollary \ref{cor-def-U}, we know that all radial solutions to \eqref{problem} are given by
  \begin{equation*}
    \lambda^{\frac {N-2}2 \nu_\gamma} \frac{|x|^{\frac {N-2}2(\nu_\gamma-1)}}{\left(1+\lambda  |x|^{(2-s)\nu_\gamma}\right)^{\frac {N-2}{2-s}}}, \text{ where } \lambda > 0.
        \end{equation*}
It is easy to check that only the one corresponding to $\lambda=1$ is kelvin invariant. Furthermore $U_\gamma\in L^{\infty}(\R^N)$ for every $\gamma\leq 0$. 
\end{proof}

Thus, we can say that $(\gamma, U_\gamma)$ is a curve of fixed points of $T$ in $(-\infty, 0]\times X$. 
 
  Our choice of $X$ has also an influence on the degeneracy of the solution $U_\gamma$. Indeed, 
consider the Fr\'echet derivative $T_v$ of the operator $T(\gamma,v)$ at the radial solution $U_\gamma$, namely
 \[ T_v(\gamma, U_\gamma):= \left(-\Delta -\frac{\gamma}{|x|^2}I\right)^{-1}\left( C_\gamma(p_s-1)\frac {U_\gamma^{p_s-2}}{|x|^s}\right).\]
 \begin{lemma}\label{lem-nondegeneracy}
 The operator $I-T_v(\gamma, U_\gamma) :\ (-\infty, 0]\times X\to X$ is invertible for $\gamma\neq \gamma_j$,  $j=1,2,\dots ,$ where $\gamma_j$ are given by \eqref{eq:gamma-j}.
 \end{lemma}
 \begin{proof}
First we notice that a function $w\in \mathrm{Ker}\left(I-T_v(\gamma, U_\gamma
)\right)$ if and only if $w$ weakly solves \eqref{eq:linearized2} and $w\in X$ by Lemma \ref{lem:T-well-def}. By Lemma \ref{lemma-lin}, we know that the unique solution to \eqref{eq:linearized2} in $D^{1,2}(\R^N)$ when $\gamma\neq \gamma_j$ is the function $Z_\gamma\in D^{1,2}(\R^N)\cap L^{\infty}(\R^N)$ defined in \eqref{eq:def-Z-gamma}. The Kelvin transform of $Z_\gamma$ gives
 \[k(Z_\gamma)(x)=-Z_\gamma(x)\]
So $Z_\gamma\notin X$ which implies that $\mathrm{Ker}\left(I-T_v(\gamma, U_\gamma)\right)=\{0\}$ when $\gamma\neq \gamma_j$.  
\end{proof}

\

\begin{remark}
\label{rmkdegen}
By Lemma \ref{lem-nondegeneracy}, when $\gamma \in(-\infty,0]$, the unique points at which the operator $I-T_v(\gamma, U_\gamma)$ is not invertible are the values $\gamma_j$ defined in \eqref{eq:gamma-j} which are explicitly computed and isolated. At these points $\gamma_j$, the degeneracy of $U_\gamma$ is given by the function $Z_{j,i}\in X$. 
  \end{remark}
\

\begin{lemma}\label{lem:comp}
  The operator $T$ is continuous with respect to $\gamma$ and is compact from $X$ into $X$ for any fixed $\gamma\in (-\infty,0]$.
\end{lemma}
\begin{proof}
First we prove that it is compact from $X$ to $X$.   Let $(v_n)_n \subset X$ be such that $\|v_n\|_X \leq C$, for some constant $C$ independent of $n$. Since $\norm{v_n}_{1,2}\leq C$, up to a subsequence, we have that $v_n\rightharpoonup    v$ weakly in $D^{1,2}(\R^N)$ and almost everywhere in $\R^N$. Since $\|v_n\|_{L^\infty}\leq C$, \eqref{eq:decay-v} implies that
  \begin{equation}\label{eq:decay-vn}
    |v_n(x)|\leq C(1+|x|)^{2-N},
  \end{equation}
 for some $C$ independent of $n$. By pointwise convergence, we deduce that $v$ satisfies the same estimate. 
Define $g_n:= T(\gamma,v_n)$ i.e. $g_n\in D^{1,2}(\R^N)$ is the weak solution to
  \begin{equation}\label{eq:gn}
    -\Delta g_n-\frac \gamma{|x|^2}g_n= C_\gamma \frac {|v_n|^{p_s-2}v_n}{|x|^s},\  \ \text{ in }\R^N\setminus \{0\}.
    \end{equation}
  Multiplying the previous equation by $g_n$ and integrating, we find
  \[\int_{\R^N} |\nabla g_n|^2 -\int_{\R^N} \frac \gamma{|x|^2}g_n^2= C_\gamma
  \int_{\R^N}\frac {|v_n|^{p_s-2}v_ng_n}{|x|^s}.\]
  The Hardy and Sobolev inequality then implies 
  \[\norm{g_n}_{1,2}^2\leq C \norm{  \frac {|v_n|^{p_s-1}}{|x|^s}}_{\frac {2N}{N+2}}\norm{g_n}_{\frac {2N}{N-2}}\leq C \norm{  \frac {|v_n|^{p_s-1}}{|x|^s}}_{\frac {2N}{N+2}}\norm{g_n}_{1,2}.\]
  Then, up to a subsequence, $g_n\rightharpoonup g$  weakly in $D^{1,2}(\R^N)$ 
  and almost everywhere in $\R^N$. So we can pass to the limit in the weak formulation of \eqref{eq:gn} obtaining that $g$ is a weak solution to
  \begin{equation}\label{eq:g}
  -\Delta g-\frac \gamma{|x|^2}g=C_\gamma \frac {|v|^{p_s-2}v}{|x|^s}, \  \ \text{ in }\R^N\setminus\{0\}.\end{equation} 
  From \eqref{eq:decay-vn} and \eqref{eq:gn}, using Lemmas \ref{lem:appendix-3} and \ref{lem-decad-w}, we have $|g_n(x)|\leq C|x|^{\beta}(1+|x|)^{2-N-\beta}$, for some $C$ independent of $n$ and $\beta=\min\{2-s,-a_\gamma\}$. By pointwise convergence, the same estimate holds for $g$. These estimates allow us to get
  \begin{eqnarray}
    &&\int_{\R^N}|\nabla g_n|^2 \ dx-\gamma\int_{\R^N}\frac {g_n^2}{|x|^2}\ dx= C_\gamma \int_\R^N\frac {|v_n|^{p_s-2}v_ng_n}{|x|^s}\ dx\to\nonumber\\
    &&  \   C_\gamma
    \int_\R^N\frac {|v|^{p_s-2}vg}{|x|^s}\ dx=\int_{\R^N}|\nabla g|^2 \ dx-\gamma\int_{\R^N}\frac {g^2}{|x|^2}\ dx,\nonumber
  \end{eqnarray}
  where the last equality follows from \eqref{eq:g}. By Lemma 5.1 of \cite{DGG}, this implies that $g_n\to g$ strongly in $D^{1,2}(\R^N)$.
 
 To finish the proof, we need to show that $\nor g_n-g\nor_{\infty}<\e$ if $n$ is large enough. To this end, observe that $g_n-g\in D^{1,2}(\R^N)$ weakly solves
  $$-\Delta (g_n- g)-\frac{\gamma}{|x|^2}(g_n- g)= C_\gamma
  \frac{|v_n|^{p_s-2}v_n- |v|^{p_s-2}v}{|x|^s} \, \hbox{ in }\R^N .$$
 As previously, using Lemma \ref{lem:appendix-3} and \ref{lem-decad-w} (see in particular \eqref{eq:decay-w}), there exists $r_0,R_0>0$ such that  $| g_n(x)- g(x)|\leq \frac \e 2$ in $\{x\in \R^N: \ |x|< r_0 \ , |x|>R_0\}$ uniformly in $n$.  Finally, since $v_n$ is uniformly bounded in $B_{R_0}\setminus B_{r_0}$, $|v_n|^{p_s-2}v_n\to |v|^{p_s-2}v$ in $L^p(B_{R_0}\setminus B_{r_0})$ for any $p$. Thus, for any $\e>0$, we get that $\sup_{B_{R_0}\setminus B_{r_0}}|g_n(x)-g(x)|\leq \frac \e 2$ provided $n$ is large enough. 
 
To prove the continuity of $T$ with respect to $\gamma$, we let $\gamma_n\in (-\infty,0)$ be such that $\gamma_n\to \gamma$ and we set $g_n=T(\gamma_n,v)$ for a given $v\in X$. As before, the Hardy and the Sobolev inequality implies that $\|g_n\|_{1,2}\leq C$, so that, up to a subsequence, $g_n\to g$ weakly in $D^{1,2}(\R^N)$ where $g$ is a weak solution to \eqref{eq:g}, namely $g=T(\gamma,v)$. The convergence of $g_n$ to $g$ in $X$ then follows exactly as in the proof of the compactness. This concludes the proof. 

\end{proof}

\

Before proving our first general bifurcation result, we introduce some notation. We recall that $O(k)$ is the orthogonal group in $\R^k$.  We define the subgroups
$\mathcal G_h$ of $O(N)$ by
\begin{equation}\nonumber
\mathcal G_h=O(h)\times O(N-h)\quad \hbox{ for }1\leq h\leq \left[\frac N2\right]
\end{equation}
where $\left[ a\right]$ stands for the integer part of $a$, for $h=1,\dots,[\frac N 2]$ .
We denote by $X^h$ the subspace of $X$ of functions invariant by the
action of $\mathcal G_h$. We will also consider the subspace of $X$ given by the functions which are invariant by the action of the orthogonal group $O(N-1)$ in  $\R^{N-1}$, which acts on the first $(N-1)$-variables and that we denote hereafter by $X^{N-1}$.

\begin{theorem}\label{teo:bif}
Fix $j\in \N$ and let $\gamma_j$ be as in \eqref{eq:gamma-j}. Then\\
i) For any $j$, there exists at least a continuum of non-radial weak solutions to \eqref{problem}, invariant with respect to $O(N-1)$, bifurcating from $(\gamma_j, U_{\gamma_j})$ in $(-\infty,0]\times X$.\\
ii) If $j$ is even, there exist at least $\big[\frac N2\big]$ continua of non-radial weak solutions to
\eqref{problem} bifurcating from $(\gamma_j,U_{\gamma_j})$ in $(-\infty,0]\times X$. The first branch is $O(N-1)$ invariant, the second is $\mathcal G_2$ invariant and so on.\\
Moreover the solutions $v_{\gamma}$ along these continua are Kelvin invariant and satisfy
$$\sup_{x\in \R^N}(1+|x|)^{N-2}|v_{\gamma}(x)|\leq C_v.$$
The bifurcation is global and the Rabinowitz alternative holds.
\end{theorem}

\begin{remark}[Rabinowitz alternative]\label{rem-alternative}
Let us recall the Rabinowitz alternative. Denote by $\Sigma^\ell$ the closure of the set
  \[\{(\gamma,v)\in(-\infty,0)\times X^\ell,\ : \ T(\gamma,v)=v,\ v\neq U_\gamma\}\]
  with $\ell=1,\dots,[\frac N2 ]$ or $\ell=N-1$ and 
by $\mathcal C_j^\ell$ the closed connected component of $\Sigma^\ell$
that contains $(\gamma_j,U_{\gamma_j})$, namely 
  the continuum of solutions to \eqref{problem} bifurcating from $(\gamma_j, U_{\gamma_j})$. The Rabinowitz alternative states that one of the following occur:
\begin{itemize}
\item[$a)$] $\mathcal C_j^\ell$ is unbounded in $(-\infty,0)\times X^\ell$;
\item[$b)$]  $\mathcal C_j^\ell$ intersects $\{0\}\times X^\ell$;
\item[$c)$]  there exists $ \gamma_h$ with $h\neq j$ such that  $( \gamma_h, U_{\gamma_h}
)\in\mathcal C_j^\ell$.
  \end{itemize}
\end{remark}

\begin{proof}[Proof of Theorem \ref{teo:bif}]
  The bifurcation result is standard when we have a compact operator, continuous with respect to the bifurcation parameter $\gamma$ (by Lemma \ref{lem:comp}) which has only isolated degeneracy points (by Lemma \ref{lem-nondegeneracy}) and such that at a degeneracy point the Morse index along the curve of solutions $(\gamma,U_\gamma)$ has an odd change. It is also well-known that the bifurcation is global and the Rabinowitz alternative holds if there is an odd change in the Morse index .

  \

So we only have to prove that at a degeneracy point $\gamma _j$ the Morse index along the curve $(\gamma, U_\gamma)$ has an odd change. To prove $i)$, we are going to work in the space $X^{N-1}$. First, we observe that $T$ maps $X^{N-1}$ into $X^{N-1}$. By Lemma \ref{lem-nondegeneracy} and Remark \ref{rmkdegen}, bifurcation in $X$ can happen only at the values $(\gamma_j,U_{\gamma_j}
)$ where the solutions of the operator $I-T_v(\gamma_j,U_{\gamma_j})$ are $Z_{j,i}$ for $i=1,\dots, \frac{(N+2j-2)(N+j-3)!}{(N-2)!j!}$. Moreover, for any $j\in \N$, there exists only one spherical harmonic $O(N-1)$-invariant. So, in $X^{N-1}$, the operator $I-T_v(\gamma_j, U_{\gamma_j}
)$ has a unique solution which means, by Corollary \ref{cor:G-morse-index-2}, that the Morse index in the space $X^{N-1}$ increases by exactly one at every value $\gamma_j$. This proves $i)$.

The proof of $ii)$ is similar except we work in the spaces $X^h$. In this case, it has been proved in \cite{SW} (see also \cite{GGT2}) that there is only one spherical harmonic $\mathcal G_h$-invariant.
\end{proof}

\

\begin{remark}
For later purpose, we now precise the uniqueness of the spherical harmonic which is $O(N-1)$-invariant for any $k\in \N$. In fact, it is given explicitly by (see \cite{G})
\begin{equation}\label{eq:spherical-harmonics-N-1}
  Y_k(\theta)=P_k^{(\frac{N-3}2,\frac{N-3}2)}(\cos \theta)
  \end{equation}
  for $\theta\in [0,\pi]$, where $P_k^{(\frac{N-3}2,\frac{N-3}2)}$ are the Jacobi Polynomials that can be written, using the Rodrigues' formula 
\begin{equation}\label{eq:Jacoby}\nonumber
P_k^{(\frac{N-3}2,\frac{N-3}2)}(z)=\frac{(-1)^k}{2^k k!}(1-z^2)^{-\frac{N-3}2}\frac{\de ^k}{\de z^k}\left( (1-z^2)^{k+\frac{N-3}2}\right)
\end{equation}
for $z\in [-1,1]$ and any $k= 0,1,\dots$. In particular we have 
$$P_1^{(\frac{N-3}2,\frac{N-3}2)}(z)=\frac {N-1}2 z$$ and $$P_2^{(\frac{N-3}2,\frac{N-3}2)}(z)= \frac{N+1}{16}((N+1)z^2-2).$$
\end{remark}

\begin{lemma}[$L^\infty$ a priori bound in $\R^N\setminus B_{1}$.]\label{Lem:incomplet}
Let $\bar\gamma<\gamma_1$. There exists $B=B(\bar\gamma)>0$ such that if $v\in X$ is a nonnegative solution of 
\begin{equation}
\label{eqreg-lem}
-\Delta v-\frac{\G}{|x|^2}v= C_{\G} 
\frac{|v|^{p_s-2}v}{|x|^s} \ \text{ in }\R^N\setminus\{0\}
\end{equation}
for some $\G\in [\bar \gamma, \gamma_1]$, 
then $\|v\|_{L^\infty(\R^N\setminus B_{1})}\le B$.
\end{lemma}
\begin{proof}
Assume there exists a sequence $(u_{\gamma_n})_{n}$ of
nonnegative solutions of \eqref{eqreg-lem} such that $\gamma_n\in [\bar \gamma, \gamma_1]$ and 
 $$\|u_{\gamma_n}\|_{L^\infty(\R^N\setminus B_1)}\to \infty.$$

One can assume that $\gamma_n \to \gamma \in [\bar \gamma,\gamma_1)$.  Let $x_n$ be a point where $u_{\gamma_n}$ achieves its maximum in $\R^N\setminus B_{1}$. Define
  \begin{equation*}
    v_{n}(y) := \mu_n u_{\gamma_n}\bigl(\mu_n^{(p_{s}-2)/2} y + x_{n}\bigr),
    \quad
    \text{where}\quad \mu_n := 1 / \norm{u_{\gamma_n}}_{L^\infty(\R^N\setminus B_{1})} \to 0.
  \end{equation*}
  Note that $v_n(0) = \norm{v_{n}}_{L^\infty(\R^N\setminus B_{1})} = 1$. Observe also that the Kelvin invariance of $u_{\gamma_n}$ implies that 
  $$\max_{B_1\setminus B_{\frac12}\ } u_{\gamma_n}\le 2^{N-2}\max_{\R^N\setminus B_{1}} u_{\gamma_n}.$$
  The function $v_n$ satisfies
    \begin{equation*}
    -\Delta v_n-\frac{\gamma_n\mu_n^{p_{s}-2}}{|\mu_n^{(p_{s}-2)/2} y + x_{n}|^2}v_n= C_{\gamma_n} 
\frac{|v_n|^{p_s-2}v_n}{|\mu_n^{(p_{s}-2)/2} y + x_{n}|^s} 
    \quad
    \text{on } B_{r_n}, 
  \end{equation*}
  where $r_n=\frac12\mu_n^{1-p_s/2}$. 
  By elliptic regularity, $(v_n)_n$ is bounded in $W^{2,r}$ and
  $C^{1,\alpha}$, $0 < \alpha < 1$ on any compact set of $B_{r_n}$. Thus, up to a subsequence and a rotation of the domain, one
concludes that
  \begin{equation*}
    v_n \to v^* \quad \text{in } W^{2,r} \text{ and } C^{1,\alpha}
    \text{ on compact sets of }
    \R^N.
  \end{equation*}
Observe also that either we can assume $x_n\to x^*$ for some $x^*\in \overline{\R^N\setminus B_{1}}$ or $|x_n|\to \infty$. 
Therefore, one has $v^* \ge 0$, $v^*(0) = 1$, $v^*\in  L^\infty$ and
  $v^*$ satisfies either
  \begin{equation*}
    -\Delta v^* = \frac1{|x^*|^s}(v^*)^{p_s-1}\quad\text{in } \R^N
    \qquad\text{ or }\qquad
      -\Delta v^* = 0 \quad\text{in } \R^{N}.
  \end{equation*}
  Liouville theorems~\cite{Gidas-Spruck81, Chen-Li, YanYan-Li03} and the fact that bounded harmonic functions on $\R^N$ are trivial imply $v^* =
  0$ which contradicts $v^*(0) = 1$.
\end{proof}
Lemma \ref{Lem:incomplet} is an incomplete result since we would have like to deduce a uniform bound in $X$. We include it for completeness to show that if a branch does explode in $L^\infty$, the trouble comes from the behaviour at the origin. We believe that a $D^{1,2}$ uniform bound would be enough to control, along a branch, the behaviour of the solutions at the origin and therefore allow to prove its unboundedness in the $\gamma$ direction. 

\begin{proposition}\label{prop:unbounded-interval}
The continua $C_j^\ell$ are given by nontrivial solutions. Moreover,
case $b)$ of Remark \ref{rem-alternative} cannot hold.
\end{proposition}
\begin{proof}
    First we show that the continua $C_j^\ell$ are given by nontrivial solutions. Assume by contradiction that there exists a sequence of functions $v_n\in C_j^\ell$ such that $v_n$ solves \eqref{problem} with $\gamma=\G_n <0$ and such that
    $\int_{\R^N}|\nabla v_n|^2 \to 0$ as $n\to \infty$. Using equation \eqref{problem} we have
    \[\int_{\R^N}|\nabla v_n|^2-\G_n\int_{\R^N}\frac{v_n^2}{|x|^2}=C_{\G_n} \int_{\R^N}\frac{v_n^{p_s}}{|x|^s}.\]
Recalling that $\G_n<0$ and using the Hardy inequality,
    \[{\int_{\R^N}\frac{v_n^2}{|x|^2}}\leq \left(\frac{N-2}2\right)^2 \int_{\R^N}|\nabla v_n|^2,\]
 we get
    \[\begin{split}
        1=&\G_n\frac{\int_{\R^N}\frac{v_n^2}{|x|^2}}{\int_{\R^N}|\nabla v_n|^2}+ C_{\G_n} 
        \frac{ \int_{\R^N}\frac{v_n^{p_s}}{|x|^s}}{\int_{\R^N}|\nabla v_n|^2}\leq  C_{\G_n}  \frac{\left(\int_{\R^N}\frac{v_n^2}{|x|^2}\right)^{\frac s2}\left(\int_{\R^N} v_n^{2^*}\right)^{\frac{2-s}2}      }{\int_{\R^N}|\nabla v_n|^2}\\
        \leq &  C_{\G_n}  S^{-\frac {N(2-s)}{2(N-2)}} \left(\frac{\int_{\R^N}\frac {v_n^2}{|x|^2} } { \int_{\R^N}|\nabla v_n|^2} \right)^{\frac s2} \left(  \int_{\R^N}|\nabla v_n|^2\right)^{\frac{2-s}{N-2}}\\
        &\leq C_{\G_n}  S^{-\frac {N(2-s)}{2(N-2)}}\left(\frac {N-2}2\right)^s\left(  \int_{\R^N}|\nabla v_n|^2\right)^{\frac{2-s}{N-2}}
        \end{split}
                \]
                where $S$ is the best Sobolev constant. As the l.h.s. converges to $0$ as $n\to \infty$, we obtain a contradiction.
    \\
Next, Proposition \ref{prop:existence} implies that the continuum $C_j$, which contains non-radial, nontrivial solutions, cannot overpass the value $\gamma_1\leq 0$ proving case $b)$ of Remark \ref{rem-alternative} cannot hold.

\end{proof}

\section{Separation of the first two branches}\label{se:4}
In this section we aim to separate two of the branches of non-radial solutions obtained in Theorem \ref{teo:bif}, case $i)$. To this end we need to identify some properties of the solutions along the continuum $ C_j^{N-1}$
 defined in Remark \ref{rem-alternative}, which are preserved and which are not satisfied by solutions along other continua. In all this section, we consider functions in $X^{N-1}$. Consider the spherical coordinates in $\R^N$, $(r,
\varphi_1,\dots,\varphi_{N-2},\theta)$ where $\varphi_1\in[0,2\pi]$,  $\varphi_i\in[0,\pi]$
$i=2,\dots,N-2$ and $\theta\in [0,\pi]$ with
\begin{equation}\label{eq:coordinates}\nonumber
\begin{array}{l}
x_1=r\sin \theta\sin \varphi_{N-2}\cdots   \sin \varphi_2\sin \varphi_1=r\sin \theta H_1(\varphi_1,\dots,\varphi_{N-2}) \\
x_2=r\sin \theta\sin \varphi_{N-2}  \cdots \sin \varphi_2\cos \varphi_1=r \sin \theta H_2(\varphi_1,\dots,\varphi_{N-2})\\
x_3=r\sin \theta\sin \varphi_{N-2}  \cdots \cos \varphi_2=r\sin \theta H_3(\varphi_1,\dots,\varphi_{N-2})\\
\qquad\cdots\\
x_{N-1}=r\sin \theta\cos \varphi_{N-2}=r \sin \theta H_{N-2}(\varphi_1,\dots,\varphi_{N-2})\\
x_N=r\cos \theta
\end{array}
\end{equation}
where $H_i(\varphi_1,\dots,\varphi_{N-2})=\frac {x_i}{r\sin \theta}=\frac {x_i}{\sqrt{x_1^2+\dots+x_{N-1}^2}}$. In these coordinates, any functions in $X^{N-1}$ depends only on $r=|x|$ and $\theta = \arccos\frac{x_N}{|x|}$. Moreover, if $v\in X^{N-1}$, we have that  
\begin{equation}\label{der1}
v(x_1,\ldots ,x_i, \ldots , x_{N-1},x_N)=v(x_1,\ldots ,- x_i, \ldots , x_{N-1},x_N),
\end{equation} 
for all $i=1,\ldots ,N-1$. Inspired by \cite{G}, we consider in $X^{N-1}$ the functions that are monotone with respect to the angle $\theta$ in a suitable interval. As we will see this angular monotonicity is preserved along continua of bifurcating solutions which will allow us to distinguish them from the others.

\

We now recall the definition of cones $\mathcal K^1_{\pm}$ and  $\mathcal{K}_{\pm}^2$
\begin{equation*}
\mathcal{K}_{\pm}^1=\left\{\begin{array}{l}v\in X^{N-1}, 
v(r,\theta) \hbox{ is non-decreasing (-increasing) in }\\ \theta \hbox{ for } (r, \theta)\in (0,\infty)\times[0,\pi]   
\end{array}\right\}
\end{equation*}
and

\begin{equation*}
\mathcal{K}_{\pm}^2=\left\{\begin{array}{l} v\in X^{N-1},  v(x',x_N)=v(x',-x_N) \hbox{ for }(x',x_N)\in \R^N\setminus\{0\},\\
v(r,\theta) \hbox{ is non-decreasing (-increasing) in } \theta \hbox{ for } (r, \theta)\in (0,\infty)\times[0,\frac \pi 2]
\end{array}\right\}.
\end{equation*}

Observe that functions in $\mathcal{K}^1_{\pm}$ are Foliated Schwarz symmetric, i.e. they are axially symmetric with respect to an axis passing through the origin and nonincreasing in the polar angle from this axis.

To deal with the case of the cone $\mathcal K^2_{\pm}$ we  denote by $X^{N-1}_{\even}$ the subspace of $X^{N-1}$ given by functions which are $even$ in $x_N$.

\

\begin{lemma}\label{lem-invariance-K} For any $\gamma\leq 0$ the operator
  $T(\gamma,-)$, defined in (\ref{def:T}), maps the cone $\mathcal K^i_{\pm}$ into
 $\mathcal K^i_{\pm}$, for $i=1,2$.
\end{lemma}
\begin{proof}
We consider the case of $\mathcal K^1_{+}$ since the proof for $\mathcal K^1_{-}$ is exactly the same. So we let $h\in \mathcal K^1_{+}$ and $u=T(\gamma,h)$, i.e. 
\begin{equation}\label{u=Tgammah}
-\Delta u  - \dfrac{\gamma}{|x|^2} u  = C_\gamma \frac {|h|^{p_s-2}h}{|x|^s}
 \  \  \ \text{ in }\R^N\setminus\{0\}. 
\end{equation}

 By the monotonicity assumption on $h$, we know that $\frac{\de
 h}{\de \theta}\geq 0$ for almost every $(r,\theta)\in (0,\infty)\times [0,\pi]$ and, by Remark \ref{rem:regularity}, $u\in W^{2,q}(\R^N)\cap C^{1,\alpha}_{loc}(\R^N\setminus\{0\})$.
One cannot take directly the derivative with respect to $\theta$ in \eqref{u=Tgammah}
 because the Laplacian does not commute with $\partial_\theta$.
 To overcome this, we let 
 $$u_{\tilde{\theta}_i } = x_{N}\partial_{x_i} u - x_{i}\partial_{x_{N}} u \in { W^{1,q}_{loc}(\R^N) }\cap C^{0,\alpha}_{loc}(\R^N\setminus\{0\}),$$ for $i=1,\ldots ,N-1$. 
 Notice that $u_{\tilde{\theta}_i }$ corresponds to the angular derivative of $u$ in the direction $\tilde{\theta}_i$ where $\tilde{\theta}_i$ is the angle associated to the rotation $\tilde{\mathcal{R}}_i$ in the plane $(x_i ,x_{N})$ centered at $0\in \R^2$ preserving all the other components. 
Observe, since $h$ is $O(N-1)$-invariant, that if $h$ is non-decreasing in  $\theta$, then $h$ is non-decreasing in ${\tilde \theta_i}$ for all $i=1,\ldots , N-1$. This follows from the fact that the rotation corresponding to $\theta$ centred at $0\in \R^N$ is a finite composition of rotations  $\tilde{\mathcal{R}}_i $ (rotation corresponding to the angle $\tilde{\theta}_i $, $i=1,\ldots ,N-1$ centred at $0\in \R^N$) and rotations $\mathcal{R}_i$ (rotation corresponding to the angle $\varphi_i$, $i=1,\ldots ,N-1$ centred at $0\in \R^N$).
 The converse is also true namely if $h$ is non-decreasing in ${\tilde \theta_i}$, for all $i=1,\ldots ,N-1$ then $h$ is non-decreasing in  $\theta$. 
 
 Direct computations using the expression of $u_{\tilde{\theta}_i }$ show that $u_{\tilde{\theta}_i }$ weakly solves
\begin{equation}
\label{equtildetheta}
-\Delta u_{\tilde{\theta}_i }  - \dfrac{\gamma}{|x|^2} u_{\tilde{\theta}_i}  = \bar C_\gamma \frac {|h|^{p_s-2}}{|x|^s}h_{\tilde{\theta}_i}
 \  \  \ \text{ in }\R^N\setminus\{0\} 
\end{equation}
for $\bar C_\gamma= C_\gamma (p_s-1)$.
Since the operator 
$$-\Delta  - \dfrac{\gamma}{|x|^2}$$
satisfies a weak maximum principle (Lemma \ref{lem:appendix-2}), we will deduce that $u_{\tilde{\theta}_i }\ge 0$ from $h_{\tilde \theta_i}\geq 0$. One needs to argue carefully since the r.h.s. does not belong to the right Lebesgue space. 
 For every $0<\e<1$, we denote by $\eta_\e(x)\in C^{\infty}_0(\R^N\setminus\{0\})$ a radial cut-off function such that
$\eta_\e=1$ for $2\e<|x|<\frac 1 \e$, $\eta_\e=0$ for $|x|\in [0,\e)\cup [\frac 2\e,\infty)$, $|\nabla \eta_\e(x)|\leq \frac 2\e$ for $|x|\in (\e,2\e)$ and $ |\nabla \eta_\e(x)|\leq 2\e$ for $|x|\in (\frac 1\e,\frac 2 \e)$. 
Let $\Psi_\e=\eta_\e (u_{\tilde \theta_i} )^-$,
where $s^-$ denotes the negative part of $s$.  
Using \eqref{equtildetheta} and the facts that $h _{\tilde{\theta}_i}\geq 0$ and $(u_{\tilde \theta_i} )^-
\leq 0$, we find that
  \[
    \begin{split}
      \int _{\R^N\setminus\{0\}}|\nabla \Psi_\e|^2 &=\int_{\R^N\setminus\{0\}}\nabla u_{\tilde \theta_i} \cdot \nabla (\eta_\varepsilon^2 (u_{\tilde \theta_i} )^- )+\int_{\R^N\setminus\{0\}}((u_{\tilde \theta_i} )^-)^2|\nabla \eta_\e|^2
      \\
      &=\gamma \int_{\R^N\setminus\{0\}}\frac{u_{\tilde \theta_i}  (u_{\tilde \theta_i} )^-}{|x|^2}\eta_\e^2+ \bar C_\gamma \int _{\R^N\setminus\{0\}}\frac{|h|^{p_s-2}}{|x|^s} h_{\tilde \theta_i} (u_{\tilde \theta_i} )^-\eta_\e^2\\
      &\ \ \ +\int_{\R^N\setminus\{0\}}((u_{\tilde \theta_i} )^-)^2|\nabla \eta_\e|^2 \\
      &  \leq \int_{\R^N\setminus\{0\}}((u_{\tilde \theta_i} )^-)^2|\nabla \eta_\e|^2 + \gamma \int_{\R^N\setminus\{0\}}\frac{((u_{\tilde \theta_i} )^-)^2}{|x|^2}\eta_\e^2.
    \end{split}\]
    
     Since by definition $|u_{\tilde \theta_i} |\leq |x| |\nabla u|$ and $u\in D^{1,2}(\R^N)$, we have
\[ \begin{split}
& \int_{\R^N\setminus\{0\}}((u_{\tilde \theta_i} )^-)^2|\nabla \eta_\e|^2 \leq 4\e^2 \frac 4{\e^2}\int_{\e<|x|<2\e}|\nabla u |^2\\
&+\frac 4{\e^2} 4\e^2
  \int_{\e^{-1}<|x|<2\e^{-1}}|\nabla u|^2=16 \int_{D_\e}|\nabla u|^2  =
  16 \int _{\R^N\setminus\{0\}}  |\nabla u|^2\chi_{D_\e},
  \end{split}\]
  where $\chi_{D_\e}$ is the characteristic function of the set 
  $D_\e:=\{x\in \R^N : \e<|x|<2\e \cup \e^{-1}<|x|<2\e^{-1}\}$. We have 
  $0\leq |\nabla u|^2\chi_{D_\e}  \leq |\nabla u|^2\in L^1(\R^N)$ and $\chi_{D_\e}\to 0$ a.e. in $\R^N$. Then the Lebesgue theorem implies that 
  \[ \int _{\R^N\setminus\{0\}}  |\nabla u|^2\chi_{D_\e}  \to 0\]
so that 
  \[\limsup_{\e\to 0}\left(\int _{R^N\setminus\{0\}}|\nabla \Psi_\e|^2 - \gamma \int_{\R^N\setminus\{0\}}\frac{((u_{\tilde \theta_i} )^-)^2}{|x|^2}\eta_\e^2\right) =0.\]
Fatou's Lemma now implies 
$$\int_{\R^N\setminus\{0\}}\frac{((u_{\tilde \theta_i} )^-)^2}{|x|^2} = 0,$$
so that $(u_{\tilde \theta_i})^- =0$ a.e. in $\R^N\setminus\{0\}$. 
This holds for every $i=1,\dots, N-1$ proving that $u$ is non-decreasing in $\theta$ i.e. $u\in \mathcal K^1_+$.

Suppose now that $h\in \mathcal{K}^2_{+}$ so that $h(x,x_N)= h(x, -x_{N})$.  First, it is easy to check that $T$ maps $X^{N-1}_{\even}$ into itself since $v=u(x,x_{N} )-u(x,- x_{N} )$ satisfies
$-\Delta v - \dfrac{\gamma}{|x|^2} v=0$ and by uniqueness, we deduce that $v=0$.
Since, by Remark \ref{rem:regularity}, $u\in C^{1,\a}_{loc}(\R^N\setminus\{0\})$, we have that $u_{x_{N}}(x,0)=0$ so  $u_{\tilde \theta_i}(x)=0$ for every $x\in \R^N\setminus\{0\}$ such that $x_N=0 $ for $i=1,\dots,N-1$.  
We are now in the same situation as the previous case except instead of working on the whole $\R^{N}$, we work on $(\R^{N})^+= \{ x\in \R^N : x_{N} > 0 \}$ and Dirichlet boundary condition on the boundary $\{x_{N}=0\}$. 
It is indeed straightforward to see that $u_{\tilde\theta_i}=0$ on  $\partial (\R^{N})^+$. Using again the cut-off function $\eta_\e$ and reasoning as in the previous part of the proof we obtain again that 
$(u_{\tilde \theta_i})^-\equiv 0$ in $(\R^N)^+$ showing $u_\theta\geq 0$. The conclusion then follows as in the previous case.
\end{proof}

When $U_\gamma$
 is an isolated fixed point for the operator $T(\gamma,\cdot)$, restricted to the space $X^{N-1}$, we can consider its index relative to the cone $\mathcal K^1_{\pm}$, 
(see \cite{D83}), which we denote by 
$\mathrm{ind}_{\mathcal K^1_{\pm}}\left(T(\gamma,\cdot),U_\gamma\right)$. It can be easily computed  when  $U_\gamma$ is non-degenerate in $X^{N-1}$. Observe that, by \eqref{eq:spherical-harmonics-N-1} this is equivalent to require that $U_\gamma$
 is non-degenerate in $X$, namely $\gamma\neq \gamma_j$ for $j=1,2\dots$. When $U_\gamma$ is an isolated fixed point for the operator $T(\gamma,\cdot)$, restricted to the space $X^{N-1}_\even$, we can consider its index relative to the cone $\mathcal K^2_{\pm}$, which we denote by 
$\mathrm{ind}_{\mathcal K^2_{\pm}}\left(T(\gamma,\cdot),U_\gamma
\right)$. Again, by \eqref{eq:spherical-harmonics-N-1} this can be computed  when $\gamma\neq \gamma_j$ for $j$ even.
In this case the characterization in Proposition \ref{lemma-lin}, see also Lemma \ref{lem-nondegeneracy}, implies the following lemma.

\begin{lemma}\label{lemma:calcoloIndexCono} Let
  $\gamma\neq \gamma_j$ be as defined in \eqref{eq:gamma-j}, for $j=1,2\dots$, then 
\[
\mathrm{ind}_{\mathcal K^1_{\pm}}\left(T(\gamma,\cdot), U_\gamma\right)=\left\{
\begin{array}{ll}
0 & \text{ if }\gamma>\gamma_1\\
- 1& \text{ if }\gamma<\gamma_1
\end{array}\right.
\]
and
\[
\mathrm{ind}_{\mathcal K^2_{\pm}}\left(T(\gamma,\cdot), U_\gamma\right)=\left\{
\begin{array}{ll}
0 & \text{ if }\gamma>\gamma_2\\
- 1& \text{ if }\gamma<\gamma_2
\end{array}\right.
\]
\end{lemma}

\begin{proof}
We denote by $T_u$ the Fr\' echet derivative of $T$ with respect to $u$. First, observe that, since $\gamma\neq \gamma_j$, $I-T_u (\gamma,U_\gamma) : X^{N-1}\longrightarrow X^{N-1}$ is invertible and the same holds in $X^{N-1}_\even$. 
Next, we claim that when $\gamma>\gamma_1$ resp. when $\gamma> \gamma_2$, the equation
\begin{equation}\label{o3}
-\Delta h-\frac \gamma{|x|^2}h -t(N-s)(N+2-2s)\nu_\gamma^2 \frac{|x|^{\nu_\gamma (2-s)-2} }{(1+|x|^{(2-s)\nu_\gamma})^2}h=0 \ \text{ in }\R^N\setminus\{0\}
\end{equation}
 does not admit a nontrivial solution $h\in X^{N-1}$ resp. $h\in X^{N-1}_\even$ which is not radial for some $t\in(0,1)$. Notice that the claim is equivalent to the so-called property $\alpha$ (see Lemma $2-(a)$ of \cite{D83}). So, applying \cite{D83}[Theorem 1], we get
$$\mathrm{ind}_{\mathcal K^1_{\pm}}\left(T(\gamma,\cdot),U_\gamma\right)=\mathrm{ind}_{X^{N-1}}\left(T_v(\gamma,U_\gamma),0\right),$$
and
$$\mathrm{ind}_{\mathcal K^1_{\pm}}\left(T(\gamma,\cdot),U_\gamma\right)=\mathrm{ind}_{X_{\even}^{N-1}}\left(T_v(\gamma,U_\gamma ),0\right).$$
Using standard results on the index (see for instance \cite{AM}), we have
$$\mathrm{ind}_{X^{N-1}}\left(T_v(\gamma, U_\gamma),0\right)=(-1)^{m^{N-1}({\gamma})},$$
and
$$\mathrm{ind}_{X_{\even}^{N-1}}\left(T_v(\gamma, U_\gamma),0\right)=(-1)^{m_{\even}^{N-1}({\gamma})}, $$
where $m^{N-1}({\gamma})$ resp. $m_{\even}^{N-1}({\gamma})$ denotes the Morse index of $U_\gamma$ in $X^{N-1}$ resp. $X^{N-1}_{\even}$. Arguing in the same way as in Lemma \ref{lem:first-radial-sing-eigen}, one can prove that $m^{N-1}({\gamma})=m_{\even}^{N-1}({\gamma})=1 $. Therefore, to finish the proof, we only need to show that the claim holds true.

First, we notice that the existence of non radial solution to \eqref{o3} is equivalent to say that zero is an eigenvalue of the problem 
\[
  -\Delta h-\frac \gamma{|x|^2}h-t(N-s)(N+2-2s)\nu_\gamma^2 \frac{|x|^{\nu_\gamma (2-s)-2} }{(1+|x|^{(2-s)\nu_\gamma})^2}h
  =\mu h \ \text{ in }\R^N\setminus\{0\}
\]
with eigenfunction in
$ X_{\rad}^\perp$ resp. $X_{\rad,\even}^\perp$ for some $t\in(0,1)$. Here $ X_{\rad}^\perp$ resp. $X_{\rad,\even}^\perp$ is the orthogonal complement to $X_\rad$ in $X^{N-1}$ resp. in $X^{N-1}_{\even}$.
We denote by $\mu_t$ the smallest eigenvalue of this problem.  By the variational characterization of the eigenvalues, $\mu_t$ is decreasing in $t$. One can also check that $\mu_t$ is continuous with respect to $t$. Moreover, by Lemma \ref{lem-first-eigenv}, $\mu_0>0$ and, by definition, $\mu_1$ is the smallest eigenvalue in $ X_{\rad}^\perp$ resp. $X_{\rad,\even}^\perp$, of the linearized operator $L_\gamma$. So, by continuity and by the decay of $\mu_t$ with respect to $t$, we deduce that there exists $t\in (0,1)$ such that $\mu_t=0$ if and only if $\mu_1<0$. From Lemma \ref{lemma-lin} and \eqref{eq:spherical-harmonics-N-1}, we have that $\mu_1<0$ in $ X^{N-1}\setminus\ X _{\rad}$
if and only if $\gamma<\gamma_1$ and $\mu_1<0$ in $ X^{N-1}_\even\setminus\ X _{\rad}$
if and only if $\gamma<\gamma_2$. This concludes the proof.
\end{proof}

\

\begin{theorem}
\label{teo-bif-a} 
The points $(\gamma_j,U_{\gamma_j} )$, $j= 1,2$, are non-radial bifurcation points from the curve of radial solutions $(\gamma, U_\gamma)$ and
the bifurcating solutions belong to the cone $\mathcal K^j_{\pm}$. Moreover, the continuum $\mathcal{C}_j^{\pm}\subset\mathcal K^j_{\pm} $, $j=1,2$, that branches out
of $(\gamma_j,U_{\gamma_j} )$ is unbounded in $(-\infty,0]\times \mathcal K^j_{\pm}$
 and the continua $\mathcal{C}_1^+$, $\mathcal{C}_1^-$,  $\mathcal{C}_2^+$ and  $\mathcal{C}_2^-$  have empty mutual intersections.
\end{theorem}

\begin{proof}
{\bf Step 1.}
{\sl Non-radial local bifurcation in $\mathcal K^j_{\pm}$.} 
We only consider the point $(\gamma_1 ,U_{\gamma_1})$ and prove the result in the cone $\mathcal K^1_{+}$,
  the other case being similar. By Lemma \ref{lemma:calcoloIndexCono}, we know that, for any $\delta>0$ small, 
\begin{equation}\label{indiceDiv-a}
\mathrm{ind}_{ \mathcal K^1_+}
\left(T(  \gamma_1-\delta ,\cdot), U_{\gamma_1-\delta}
\right)
\neq
\mathrm{ind}_{\mathcal K^1_+}
\left(T(\gamma_1+\delta ,\cdot),U_{\gamma_1+\delta}
\right).
\end{equation}
We  assume by contradiction that  $(\gamma_1,U_{\gamma_1 })$ is not a
bifurcation point in $(-\infty,0)\times \mathcal K^1_+$. Then we can find $\delta>0$ and a neighborhood $\mathcal O$ of  $\{(\gamma,U_\gamma ): \gamma\in ( \gamma_1-\delta, \gamma_1+\delta)\} $ in $(\gamma_1-\delta, \gamma_1+\delta)\times\mathcal K^1_+$ such that $v-T(\gamma,v)\neq 0$ for every $(\gamma,v)$ in $\mathcal O$ different from $(\gamma,U_\gamma)$. We can choose $\delta>0$ such that \eqref{indiceDiv-a} holds.
Let $\mathcal O_\gamma:=\{v\in \mathcal K^1_+\ : \ (\gamma,v)\in \mathcal O\}$. It follows that  there are no solutions to  $v-T(p,v)= 0$ on $\cup_{\gamma\in  (\gamma_1-\delta, \gamma_1+\delta) }\{\gamma\} \times \partial  \mathcal{O}_\gamma$ 
and there is only the radial solution $(\gamma,U_\gamma)$ in $\left(\{ \gamma_1-\delta\}\times \mathcal{O}_{  \gamma_1-\delta }\right)\cup  \left(\{ \gamma_1+\delta \}\times \mathcal{O}_{\gamma_1+\delta }\right)$. By the  homotopy invariance of the fixed point index in the cone, see \cite{D83}, we have that 
\[
\mathrm{ind}_{\mathcal K^1_+}\left(T(\gamma,\cdot),U_\gamma\right) \ \ \hbox{ is constant for } \  \gamma\in(  \gamma_1-\delta,  \gamma_1+\delta),
\]
which is in contradiction with \eqref{indiceDiv-a}.
This proves the local bifurcation.

\

We point out that the bifurcating solutions belong to $\mathcal K^1_+$ since $T$ maps the cone in itself (by Lemma  \ref{lem-invariance-K}) and are non-radial for $\gamma$ close to $\gamma_1$
since
$U_\gamma$ is radially non-degenerate in $X$.

\

{\bf Step 2.}
{\sl Global bifurcation and Rabinowitz alternative.}
As for the Step 1, we only consider the case $\gamma=\gamma_1$ and $\mathcal K^1_+$. Following  \cite{G10}, we let 
$$\mathcal S\ : =\left\{(\gamma,U_\gamma): \gamma\in (-\infty,0)\right\} \subseteq (-\infty,0)\times \mathcal K^1_+$$
  be the curve of radial solutions, $\Sigma_1^+$ be the closure of the set
$$  \{(\gamma,v)\in  \left((-\infty,0)\times \mathcal K^1_+\setminus \mathcal S\right) : v \text{ solves }v-T(\gamma,v)=0\},$$ 
and $\mathcal C_1^+$ be the closed connected component of $\Sigma_1^+$ that contains $( \gamma_1 ,  U_{\gamma_1})$ (which is nonempty by Step 1).
Assume by contradiction that
the Rabinowitz alternative, namely one of the following, does not occur:
\begin{itemize}
\item[$i)$] $\mathcal C_1^+$ is unbounded in $(-\infty,0)\times \mathcal K^1_+$;
\item[$ii)$]  $\mathcal C_1^+$ intersects $\{0\}\times \mathcal K^1_+$;
\item[$iii)$]  there exists $\gamma_{k}$ with $k\neq 1$ such that  $(\gamma_{k} ,U_{\gamma_{k}}
)\in \mathcal C_1^+\cap \mathcal S$.
  \end{itemize}
We have already observed in Proposition \ref{prop:unbounded-interval} that $ii)$ cannot hold. 
Then, as in Step 2 in the proof of \cite[Theorem 3.3]{G10}, we can construct a suitable neighborhood $\mathcal O$  of $\mathcal C_1^+$ in $\mathcal K^1_+$ such that $\partial \mathcal O\cap \Sigma_1^+=\emptyset$, $\mathcal O\cap \mathcal S\subset (\gamma_1-\delta, \gamma_1+\delta)\times  \mathcal K^1_+$, for $\delta$  such that $0<\gamma_1-\delta< \gamma_1+\delta<\gamma_2$. Moreover, there exists $c_0>0$ such that  $\| v-U_\gamma\|_{ X}\geq c_0$ for $(\gamma,v)\in \mathcal O$ such that $|\gamma-\gamma_1|\geq \delta$. Then we can follow the proof of Step 3 and Step 4 in \cite[Theorem 3.3]{G10}, recalling now that, for $\Lambda_c:=\{(\gamma,v)\in (-\infty,0)\times X^{N-1} :  \| v-U_\gamma\|_{X}< c\}
$
one has
\[\mathrm{deg}_{\mathcal K^1_+}(I-T(\gamma_1\pm \delta,\cdot), \left(\mathcal O\cap \Lambda_c\right)_{\gamma_1\pm \delta},0)=\mathrm{ind}_{\mathcal K^1_+}(T(\gamma_1\pm \delta,\cdot),u_{\gamma_1\pm \delta})\]
for any $c<c_0$.
The fixed point index relative to the cone $\mathcal K_+^1$ can then be computed in $\gamma_1\pm \delta$  and it assumes either the value $0$ or $-1$ (by Lemma \ref{lemma:calcoloIndexCono}). The proof of Step 3 and 4 of \cite[Theorem 3.3]{G10} can be repeated and so we get a contradiction.
\\
We can also adapt the proof of \cite[Proposition 2.3]{G}, again using the degree in the cone $\mathcal K_+^1$ which is, as already observed,  either $0$ or $-1$ in a neighborhood of the isolated (in $X^{N-1}$) solution $U_\gamma$. 
The main difference is that, in the final part of the proof of  \cite[Proposition 2.3]{G} we now obtain, following the notations of \cite{G}, that
\[\mathrm{deg}_{\mathcal K_+^1}(S_r(\gamma,v), \mathcal O\cap B_r(\gamma_1,U_{\gamma_1}
),0)= -1 .\]
Following the rest of the proof of \cite[Proposition 2.3]{G} we get that, whenever $\mathcal C_1^+$ is bounded, the number of degeneracy points $\gamma_k$, that belong to $\mathcal C_1^+$ and
at which the index of the operator $T(\gamma,\cdot)$ in the cone $\mathcal K_+^1$ changes, has to be even. Lemma \ref{lemma:calcoloIndexCono} then implies that the value $\gamma_1$ is the unique at which a change in the index $\mathrm{ind}_{\mathcal K_+^1}\left(T(\gamma,\cdot),U_\gamma\right)$ appears, meaning that $\mathcal C_1^+$ cannot be bounded.  
 
 \
 
{\bf Step 3.}
{\sl Conclusion.} 
As already pointed out, the bifurcating solutions along the continuum $\mathcal C_j^{\pm}$ are not radial for $\gamma$ close to $\gamma_j$, and belong to the cone $\mathcal K_{\pm}^j$. Since $\mathcal K^1_+\cap \mathcal K^1_-=X_\rad$, $\mathcal K^1_{\pm} \cap \mathcal K^2_{\pm}=X_\rad$ and $U_{\gamma_j}$
 does not possess any radial degeneracy point different from $\gamma_0=\frac{(N-2)^2}4$, the continua $\mathcal{C}_1^+$, $\mathcal{C}_1^-$,  $\mathcal{C}_2^+$ and  $\mathcal{C}_2^-$  have empty intersections.
\end{proof}

\

\begin{remark}
  Theorem \ref{teo-bif-a} proves that at the points $\gamma_i,\ i=1,2$, we have two continua of bifurcating solutions. One belongs to $\mathcal K^i_+$ and the other one is in  $\mathcal K^i_-$ for $i=1,2$. 
 Solutions in $\mathcal{C}_1^{\pm}$ exhibit a one peak profile and have their maximum point either in the 
  north pole of $S^{N-1}$ or in the south pole.  
Solutions in $\mathcal{C}_2^-$ have their maximum point in the north pole of $S^{N-1}$ and, for symmetry reasons, exhibit then a two peaks profile.
  Finally, solutions in $\mathcal{C}_2^+$ have their maximum points all along the equator of $S^{N-1}$ and have therefore a manifold, $S^{N-2}$, of maxima. 
  \end{remark}

As observed in the proof of Theorem \ref{teo-bif-a}, each time we have an odd change in the Morse index of $U_\gamma$ at $\gamma_j$, we have a bifurcation phenomenon. Of course many different symmetries can be considered. Some of them have been exploited in \cite{AG18-bif} to prove a bifurcation result from a sign-changing solutions. Due to the expression of the eigenfunctions in \eqref{decomposition-eigenfunctions}, all the symmetries that one can consider are the ones of the spherical harmonics $Y_k(\theta)$, see also Corollary \ref{cor:G-morse-index-2} for the exact computation of the change of the Morse index at a value $\gamma_j$. 

  \section{Minimization in symmetry classes}\label{se:6}
      In this section, we find solutions to \eqref{problem} with the same symmetries of the ones in Section \ref{se:4} minimizing a suitable functional associated with equation \eqref{problem}. 
We include this approach to support the conjecture that the corresponding branches of solutions obtained in Theorem \ref{teo-bif-a} exists for all $\gamma<\gamma_1$, resp. $\gamma<\gamma_2$.
      To this end we define the functional $F: D^{1,2}(\R^N)\to \R$ 
         \begin{equation}\label{eq:functional-F}
        F(u):=\frac 12 \int_{\R^N}|\nabla u|^2-\frac \gamma 2\int_{\R^N}\frac{ u^2}{|x|^2}-\frac {C_\gamma}{p_s} \int_{\R^N}\frac{ |u|^{p_s}}{|x|^s}
          \end{equation}
          which is of class $C^1$ and we recall that if $u$ is a critical point of $F$, then $u$ is a weak solution to \eqref{problem}.\\
          The functional $F$ is not bounded by below in $D^{1,2}$ and so, using a classical procedure, we restrict it on the Nehari set
           \[\mathcal N:=
            \{u\in D^{1,2}(R^N) \ u\neq 0:  \int_{\R^N}|\nabla u|^2-\gamma\int_{\R^N}\frac{ u^2}{|x|^2}- C_\gamma \int_{\R^N}\frac{ |u|^{p_s}}{|x|^s}  =0\} \]
         Now define
          \begin{equation}\label{eq:d}
      d_\gamma:=\inf_{u\in \mathcal N}F(u).
            \end{equation}
          By classical methods we prove the following result.
          \begin{theorem}
            \label{teo-d-attained}
            The infimum $d_\gamma$ on $\mathcal N$ is nonnegative, every minimizer $u_\gamma\in \mathcal N$ is a critical point of $F$, it is positive and weakly solves \eqref{problem}. Moreover every minimizer $u_\gamma$ is $O(N-1)$-invariant (in a suitable coordinate system) and foliated Schwarz symmetric. Further, when $\gamma<\gamma_1$, $u_\gamma$ is non-radial and belongs to $\mathcal K^1_{\pm}$.
          \end{theorem}

\begin{proof}
 The existence part of the theorem and the positivity of minimizers are very classical and we omit it. Now, we focus on the symmetry of  $ u_\gamma$. Every minimizer $u_\gamma$ is a minimum of $F(u)$ on the Nehari manifold $\mathcal N$ which has codimension one and so the Morse index of $u_\gamma$ is $1$. Then, as in Proposition 2.10 of \cite{GPW} there exists a direction $e\in S^{N-1}$ such that, denoting by $S(e):=\{x\in \R^N: x\cdot e>0\}$, it holds
     \[\inf_{\psi\in C^1_0(S(e))}Q(\psi,\psi)      \geq 0\]
   where $$Q(\psi,\phi)=\int_{\R^N} (\nabla \psi \nabla \phi - \gamma \phi \psi - C_\gamma (p_s -1) |u_\gamma |^{p_s-2}\phi \psi ) dx .$$ 
We denote by $\sigma_e$ the reflection with respect to the hyperplane $x\cdot e=0$. Then we consider two cases depending if $u(x)-u(\sigma_e(x))=0$ or not. If  $u(x)-u(\sigma_e(x))\equiv 0$ in $S(e)$, then the foliated Schwarz symmetry of $u_\gamma$ follows by Proposition 2.5 of \cite{GPW}.  In the other case, the proof of Theorem 1.4 in \cite{GPW} gives that $u(x)-u(\sigma_e(x))$ does not change sign in $S(e)$. The foliated Schwarz symmetry of $u$ then follows by Proposition 2.8 of \cite{GPW}. 
By definition of foliated Schwarz symmetry, we get that $u_\gamma$ depends only on $r=|x|$ and on one angle $\theta=\arccos  (\frac {x}{|x|}\cdot e)$, with $e\in S^{N-1}$ in which it is monotone. This proves that $u_\gamma$ is $O(N-1)$-invariant and, up to rotation, belongs to $ \mathcal {K}^1_{\pm}$.  
  Finally, we prove that every minimizer $u_\gamma$ is non-radial for $\gamma<\gamma_1$. 
Following the proof of \cite{GPW}, we have that either $ u_\gamma$ is radial or, up to a change of coordinates, it is strictly decreasing with respect to $\theta$ and therefore belongs to the interior of the cone $\mathcal K_1$. The fact that $u_\gamma$ is not radial when $\gamma<\gamma_1$ follows by Morse index considerations. As already observe the Morse index in $D^{1,2}(\R^N)$ of every minimizers $u_\gamma$ is one. But when $\gamma<\gamma_1$, the Morse index of the radial solution $U_\gamma$ is greater or equal than $N+1$ from \eqref{eq:morse-index}.  Then $u_{\gamma}\neq U_{\gamma}$ and the proof is completed.
         \end{proof}

         \

         Next we turn to the case of the cone $\mathcal K_2$ considered in Section \ref{se:4} and we restrict the functional $F$ in \eqref{eq:functional-F} to the space $D^{1,2}_{N-1, \even}$, given by $D^{1,2}$ functions which are $O(N-1)$-invariant with respect $(x_1, \dots,x_{N-1})$ and even in $x_N$.
         We denote by
             \[\mathcal N^{\even}:=
               \{u\in D^{1,2}_{N-1, \even} \ u\neq 0:  \int_{\R^N}|\nabla u|^2-\gamma\int_{\R^N}\frac{ u^2}{|x|^2}- C_\gamma\int_{\R^N}\frac{ |u|^{p_S}}{|x|^s}  =0\} \]
             and 
   \begin{equation}\label{eq:d-even}
             d_\gamma^{\even}:=\inf_{u\in \mathcal N^{\even}}F(u).
            \end{equation}

     \begin{theorem}
            \label{teo-d-even-attained}
            The infimum $ d_\gamma^{\even}$ on $\mathcal N^{\even}$ is nonnegative, every minimizer $u_{\gamma}^{\even}\in \mathcal N^{\even}$ is a critical point of $F$, it is positive and weakly solves \eqref{problem}. Moreover, when $\gamma<\gamma_1$ $u_{\gamma}^\even \neq u_{\gamma}$ and every minimizer $u_{\gamma}^\even$ is non-radial for every $\gamma<\gamma_2$.
          \end{theorem}

          \begin{proof}      
As previously, we omit the existence part of the proof. Following the former proof, one can show that $u_{\gamma}^\even$ is $O(N-1)$-invariant and even in $x_N$. So it belongs to $\mathcal K^1_{\pm}$ if and only if it is radial. Then Theorem \ref{teo-d-attained} implies that $u_{\gamma}^\even\neq u_{\gamma}$, for $\gamma<\gamma_1$, since $u_{\gamma}$ is strictly decreasing in $\theta$. Finally the fact that $ u_{\gamma}^\even$ is not radial when $\gamma<\gamma_2$ follows by Morse index considerations. Indeed, as before, the Morse index of  every minimizers $u_{\gamma}^\even$ in the space  $D^{1,2}_{N-1,\even}$ is one. When $\gamma<\gamma_2$, the Morse index of the radial solution $U_{\gamma}$ in the space  $D^{1,2}_{N-1,\even}$ is greater or equal than $2$, from \eqref{eq:morse-index} and Corollary \ref{cor:G-morse-index-2}. Recall indeed that corresponding to $j=1$ there are no eigenfunction $Y_j(\theta)$ of the Laplace Beltrami operator in $D^{1,2}_{N-1,\even}$, while corresponding to  $j=2$ the spherical harmonic $P_2^{(\frac {N-3}2,\frac {N-3}2)} (\theta)$ in \eqref{eq:spherical-harmonics-N-1} belongs to $D^{1,2}_{N-1,\even}$. Thus $u_{\gamma}^\even\neq U_{\gamma}$.             
            \end{proof}

\section{Other bifurcation results}\label{se:5}
In this section, we obtain other continua of bifurcating solutions by exploiting different monotonicity properties along other bifurcating branches. We start from the expression of the eigenfunctions $Y_k$ of the Laplace-Beltrami operator in ${\mathbb S}^{N-1}$. In spherical coordinates, they can be written as
\begin{equation}\label{N=3}
Y_k(\varphi_1,\theta)=\sum\limits_{\ell=0}^k  P_k^{\ell}(\cos\theta)\left(A_{\ell} \cos \ell\varphi_1+ B_{\ell}\sin \ell\varphi_1\right), 
\end{equation}
in dimension $N=3$, where $P_k^{\ell}$ are the associated Legendre polynomials, and as 
\begin{equation}\label{N>3}
\begin{split}
  Y_k(\varphi,\theta)
  & =\!\!\!\mathop{\sum\limits_{\ell=0}^{k} \prod\limits_{j=2}^{N\!-\!2}}\limits_{\substack{\ell=i_0\le i_1\dots i_{N-3}\le k}}\!\!\!
 G_{k}^{i_{N\!-\!3}}(\cos\theta,\!N-\!3) G_{i_{j-1}}^{i_{j\!-\!2}}(\cos\varphi_{j},j-2) \\
 &\ \ \ \ \ \ \ \left(A_{\ell}^{i_1\dots i_{N\!-\!3}} \cos \ell\varphi_1+ B_{\ell}^{i_1\dots i_{N\!-\!3}}\sin \ell\varphi_1\right), 
 \end{split}\end{equation}
in dimension $N>3$, where $G_{i}^{0}(\cdot,j)$ are the Gegenbauer Polynomials. 

We point out that $G_{i}^{\ell}(\omega,0) =  P_{i}^{\ell}(\omega) $ and $G_i^0(-, j)=P^{(\frac j2, \frac j2)}_i(-)$ are the Jacobi Polynomials as in \eqref{eq:spherical-harmonics-N-1}. In particular, there are eigenfunctions $Y_k$ periodic  with respect to the angle $\varphi_1$ of periodicity $\frac {2\pi}j$, for $j\in \N_0$, as far as $k\geq j$. When $j=k$, they are given by:
\[ \left(A_{k} \cos k\varphi_1+ B_{k}\sin k\varphi_1\right)(\sin \theta)^k\prod\limits_{j=2}^{N\!-\!2}(\sin \varphi_{j})^k .\]
We next introduce some functional spaces having the same symmetry properties. Let $O_j$, $j\in \N_0$, be the subgroup of rotations of angle $\frac {2\pi}j$ in the $(x_1,x_2)$-plane. Denote by $\tau$ the reflection with respect to the hyperplane $x_2=0$, namely $\tau(x_1,x_2,\dots,x_N)= (x_1,-x_2,\dots,x_N)$. For any $k\in\N_0$, we define
\begin{equation}\label{g-j}
\mathcal{G}_j\subset O(N) \ \text{ the subgroup generated by the elements of $O_j$ and by $\tau$ },
\end{equation}
and 
\begin{equation}
\label{Hj}
D^{1,2}_{j}\ : =\{v\in D^{1,2}(\R^N)\ \text{ such that } v(g(x))=v(x), \  \  \forall g\in \mathcal{G}_j, \  \forall x\in \R^N  \}.
\end{equation}
The functions in the spaces $ D^{1,2}_{j}$ clearly possess the following invariances (in polar coordinates $(x_1,x_2)=(\rho\cos \psi,\rho\sin \psi)$), with $\rho^2=x_1^2+x_2^2$ and $\psi\in[0,2\pi]$:
\begin{eqnarray} \label{fi00}
&v(\rho,\psi,x_3,\dots,x_N)=v(\rho,2\pi-\psi,x_3,\dots,x_N),\\\label{fi0}
&v(\rho,\psi,x_3,\dots,x_N)=v(\rho, \psi+\frac{2\pi}j,x_3,\dots,x_N) ,
\end{eqnarray}
and, by combining the above two relations,
\begin{equation}\label{fi1}
v(\rho,\frac \pi j+\psi, x_3,\dots,x_N)=v(\rho,\frac \pi j-\psi, x_3,\dots,x_N)
\end{equation}
for every $(\rho, \psi ,x_3,\dots,x_N)\in (0,\infty )\times [0,2\pi ] \times \R^{N-2}$. Let 
\[X_j=X\cap D^{1,2}_{j.}\]
We will prove that, for any $j\in\N$, there exists an unbounded continuum of solutions in $X_j$ bifurcating from the radial one.
In order to identify different continua, we restrict the operator $T$ to  suitable cones $\tilde{\mathcal K}^j _{\pm}\subset X_j$ defined, similarly as in  \cite{D2}, by imposing some  angular monotonicity on the $\mathcal G_j$-symmetric functions. Hence, for $j\in\mathbb N$,  we define the cone
 \begin{equation*}
\tilde{\mathcal K}^j_{\pm}\ : \ =
\left\{\begin{array}{l}v\in X_j, 
v(\rho,\psi, x_3,\dots,x_N) \hbox{ is non-decreasing (-increasing) in }\\ \psi \hbox{ for } (\rho, \psi)\in (0,\infty)\times[0,2\pi/j) , (x_3,\dots,x_N)\in\R^{N-2}
\end{array}\right\}.
\end{equation*}

By definition, $X_{\rad}\subset \tilde{\mathcal K}^j_{\pm}\subset X_j$ for any $j\geq 1$. Moreover, denoting by $\mathcal X_\psi$ the subspace of functions in $X$ that do not depend on the angle $\psi$,  we see that $\mathcal X_\psi\subset X_j$ for any $j\geq 1$ and
\begin{equation}\label{quelloCheGuadagnoConConi}
 \tilde{\mathcal K}^j_+\cap \tilde{\mathcal K}^j_-\subseteq\mathcal{X}_{\psi} \ ; 
\quad\tilde{\mathcal K}^j_{\pm}\cap \tilde{\mathcal K}^h_{\pm}\subseteq\mathcal{X}_{\psi}\
{\text{ and } \tilde{\mathcal K}^j_{+}\cap \tilde{\mathcal K}^h_{-}\subseteq\mathcal{X}_{\psi}.
} 
\end{equation}

Observe that, for any fixed $\gamma$, the operator $T(\gamma, \cdot)$ is compact and continuous in $\gamma$. Also its restriction to the subspaces $X_j$, $j\geq 1$, is compact and continuous in $\gamma$.\\

\

First we have
\begin{lemma} \label{lemma-operator-in-cone-2}
  The operator $T(\gamma, \cdot)$ maps $\tilde{\mathcal K}^j_{\pm}$ into $\tilde{\mathcal K}^j_{\pm}$.
\end{lemma}
\begin{proof}
  Let $h\in X_j$ and $z=T(\gamma,h)$. It is easy to see that $T(\gamma, \cdot)$ maps $X_j$ into $X_j$. Next, we consider the case of $h\in \tilde{\mathcal K}^j_{+}$. 
  By Remark \ref{rem:regularity}, we know that $z\in W^{2,q}(\R^N)\cap C^{1,\alpha}_{loc}(\R^N\setminus\{0\})$ so $z_{\psi}:=\frac{\partial z}{\partial \psi}\in {W^{1,q}_{loc}(\R^N)}\cap C^{0,\alpha}_{loc}(\R^N\setminus\{0\})$. For any $\varphi\in C^{\infty}_0(\R^N\setminus\{0\})$ 
one can check that $z_\psi$ weakly solves
  \begin{equation}\label{eq:z-psi}
  \int_{\R^N} \nabla z_\psi\nabla \varphi-\gamma\int_{\R^N}|x|^{-2}z_\psi\varphi=\bar C_\gamma
  \int_{\R^N} \frac {|h|^{p_s-2}}{|x|^s}h_\psi\varphi ,\end{equation}
  where $h_\psi:=\frac{\partial v}{\partial \psi}$. 
  
Arguing as in the proof of Lemma \ref{lem-invariance-K} we deduce that
  \[\int_{\Sigma_j} |x|^{-2}(z_\psi^-)^2=0,\]
  where $\Sigma_j:=\{x\in \R^N\setminus\{0\}\ , 0< \psi< \frac{2 \pi} j\}$,
  showing that $z_\psi\geq 0$.
\end{proof}

\

As in Section \ref{se:4}, we have 
\begin{lemma}\label{lemma:calcoloIndexCono2} 
Let
  $\gamma\neq \gamma_i$, $i\in \N$, be defined as in \eqref{eq:gamma-j}. Then, for $j\geq 1$, we have 
\[
\mathrm{ind}_{\tilde{\mathcal K}^j_{\pm}}\left(T(\gamma,\cdot),U_\gamma\right)=\left\{
\begin{array}{ll}
0 & \text{ if }\gamma>\gamma_j ,\\
(- 1)^{m^j(\gamma)}& \text{ if }\gamma<\gamma_j ,
\end{array}\right.
\]
where $m^j(\gamma)$ denotes the Morse index of $U_\gamma$ in the symmetric space $X_j$.
\end{lemma}
\begin{proof}
 The proof is very similar to the one of Lemma \ref{lemma:calcoloIndexCono}. Just observe that, from \eqref{N>3}, there exists an eigenfunction in $X_j$ (associated with a negative eigenvalue) which depends on the angle $\psi$ if and only if $\gamma <\gamma_j$.
\end{proof}

Again we have:
\begin{theorem}
\label{teo-bif-2} 
The points $(\gamma_j,U_{\gamma_j} )$, $j\geq 1$ are non-radial bifurcation points from the curve of radial solutions $(\gamma, U_\gamma)$ and the bifurcating solutions belong to the cone $\tilde{\mathcal K}^j_{\pm}$. Moreover, the continuum $\tilde{\mathcal{C}}_j^{\pm}\subseteq \tilde{\mathcal K}^j_{\pm}
$, $j\in \N$, that branches out
of $(\gamma_j, U_{\gamma_j})$ is unbounded in $(-\infty,0)\times \tilde{\mathcal K}^j_{\pm}$
and the continua $\tilde{\mathcal C}_j^{+}$, $\tilde{\mathcal C}_j^{-}$, $\tilde{\mathcal C}_h^{+}$, $\tilde{\mathcal C}_h^{-}$ can intersect only in $\mathcal X_{\psi}$.
\end{theorem}

\begin{proof}
The proof is a straightforward adaptation of the proof of Theorem \ref{teo-bif-a} and we omit it. The last property follows from \eqref{quelloCheGuadagnoConConi}.
\end{proof}

\

\appendix

\section{}\label{se:7}

In this appendix, we collect some general results on solutions $w \in D^{1,2}(\R^N) $ to
\begin{equation}\label{eq:standard-linear}
  -\Delta w-\frac{\gamma}{|x|^2}w=f(x) \ \ \text{ in }\R^N\setminus\{0\},
  \end{equation}
where $\gamma \leq 0$ and $N\geq 3$.
\begin{lemma}\label{lem:appendix-1}
For every function $f(x)\in L^{\frac{2N}{N+2}}(\R^N)$, \eqref{eq:standard-linear} admits a unique solution in $D^{1,2}(\R^N)$.
\end{lemma}
\begin{proof} See \cite[Lemma 5.2]{DGG}. \end{proof}
\begin{lemma}[Maximum Principle]\label{lem:appendix-2}
Let $f(x)\in L^{\frac{2N}{N+2}}(\R^N)$ such that $f(x)\leq 0$ almost everywhere in $\R^N\setminus \{0\}$ and $w\in D^{1,2}(\R^N)$ be a solution to \eqref{eq:standard-linear}. Then $w\leq 0$ in $\R^N\setminus\{0\}$. 
\end{lemma}
\begin{proof}
  By definition of $w$, we have that
  \[\int_{\R^N}\nabla w\nabla \psi \, dx-\gamma \int_{\R^N} \frac{w\psi}{|x|^2}\, dx =\int_{\R^N}f(x)\psi(x) \ dx\]
  for every $\psi\in D^{1,2}(\R^N)$. Choosing $\psi=w^+$ (where $a^+$ denotes the positive part of $a$, namely $a^+=\max\{a(x),0\}$) and using that $f\leq 0$, we have
\[\int_{\R^N}|\nabla w^+|^2 \, dx-\gamma \int_{\R^N} \frac{(w^+)^2}{|x|^2}\, dx =\int_{\R^N}f(x)w^+ \ dx\leq 0.\]
This implies that
\[\int_{\R^N} |\nabla w^+|^2=0,\]
and $w^+=0$ almost everywhere in $\R^N\setminus\{0\}$.
\end{proof}

As a consequence of the Maximum Principle, the Comparison Principle also holds, namely
\begin{lemma}[Comparison Principle]\label{lem:appendix-3}
Let $f_1, f_2\in L^{\frac{2N}{N+2}}(\R^N)$ such that
  $f_1\leq f_2$ almost everywhere in $\R^N\setminus \{0\}$. Then $w_1\leq w_2$ where $w_i$ is the unique weak solutions to \eqref{eq:standard-linear} corresponding to $f_i$.
\end{lemma}

\

Finally, we obtain a decay estimate for solutions to a suitable linear equation. 
\begin{lemma}\label{lem-decad-w}Let $w\in D^{1,2}(\R^N)$ be the unique weak solution to
  \begin{equation}\label{w}
  -\Delta w- \frac \gamma{|x|^2} w=\frac 1{|x|^s (1+|x|)^{N+2-s}} \text{ in }\R^N\setminus\{0\},\end{equation}
  with $N\geq 3$, $\gamma\leq 0$ and $s\in[0,2)$. Then $w\in L^{\infty}(\R^N)$ and 
  \begin{equation}\label{eq:estimate-w}
  \begin{split}
    &    |w(x)||x|^{N-2} \leq C\max \{      |x|^{a_\gamma},|x|^{-2}\}
     \ \text{ as  } |x|\to \infty, \\
    &      |w(x)| \leq C  \max\{|x|^{-a_{\gamma}},|x|^{2-s}\}
    \ \text{ as }|x|\to 0.
         \end{split}
       \end{equation}
       Moreover when $\gamma<0$ it holds
       \begin{equation}\label{eq:decay-w}
 |w(x)||x|^{N-2}=o(1)  \ \text{ as  } |x|\to \infty \ \ \text{ and } \ \   |w(x)| =o(1)   \ \text{ as }|x|\to 0.
         \end{equation}
\end{lemma}
  \begin{proof}
    First observe that, since the r.h.s. of \eqref{w} belongs to $L^{\frac{2N}{N+2}}(\R^N)$, there exists a unique solution $w\in D^{1,2}(\R^N)$ which satisfies  $w> 0$ almost everywhere in $\R^N\setminus \{0\}$ by the Maximum Principle. Moreover, $w$ is radial by uniqueness. 
    Let $\bar w:=r^a w(r)$ with $a=\frac {N-2}2(1-\nu_\gamma)$ as in \eqref{a-gamma}. The function $\bar w$ solves weakly
  \[-(r^{N-1-2a} w')'=\frac {r^{N-1-a-s}}{(1+r)^{N+2-s}} \ \text{ in }\ (0,+\infty).\]
  For any $r_0>0$, we have 
  \begin{equation}\label{eq:w}
  -\bar w'(r)=C_{r_0}r^{1+2a-N}+r^{1+2a-N}\int_{r_0}^r \frac {t^{N-1-a-s}}{(1+t)^{N+2-s}}\, dt ,\end{equation}
where $C_{r_0}=-r_0^{N-2a-1}\bar w'(r_0)$. From \eqref{eq:w}, we get that
\[ -\bar w'(r)\leq C_{r_0}r^{1+2a-N}+r^{1+2a-N}\int_{r_0}^r r^{-3-a}\, dr,\]
which gives, recalling \eqref{a-gamma},
\[-\bar w'(r)\leq  
\begin{cases}
 \left( C_{r_0} + \frac 1{2+a}r_0^{-(2+a)}\right)  r^{1+2a-N}- \frac 1{2+a}r^{a-1-N} & \text{ when } \gamma\neq -2N,\\
    \left( C_{r_0} -\log r_0+\log r       \right)  r^{1+2a-N}  & \text{ when } \gamma=-2N.
  \end{cases}\]
   Since $ w\in D ^{1,2}(\R^N)$, from Ni's radial Lemma (see \cite{Ni}), we know that $ w(r)\leq C_w r^{\frac{2-N}2}$, so that $\bar w(r)\leq Cr^{a-\frac {N-2}2}=Cr^{-\frac{ N-2}2\nu_\gamma}\to 0$, as $r\to +\infty$.
Integrating $\bar w'(r)$ from $r$ to $+\infty$ yields to
\[\bar w(r)\leq\begin{cases}
  \frac {C_{r_0}}{N-2a-2}r^{2+2a-N}+
    \frac {1}{(a-N)(2+a)}r^{a-N} & \text{ when }\gamma< -2N,\\
    \frac 1 {N-2a-2} r^{2+2a-N}\left( C_{r_0}-\log r_0 +\frac 1{N-2-2a}+\log r\right)
    & \text{ when }\gamma=-2N,\\
    \left( \frac{C_{r_0}}{(N-2a-2)}+\frac {r_0^{-(a+2)}}{(a+2)(N-2a-2)}\right) r^{2+2a-N}& \text{ when }\gamma>-2N,
\end{cases}\]
which implies that
\[w(r) \leq\begin{cases}
  \frac {C_{r_0}}{N-2a-2}r^{2+a-N}+
    \frac {1}{(a-N)(2+a)}r^{-N} & \text{ when }\gamma< -2N,\\
    \frac 1 {N-2a-2} r^{2+a-N}\left( C_{r_0}-\log r_0 +\frac 1{N-2-2a}+\log r\right)
    & \text{ when }\gamma=-2N,\\
    \left( \frac{C_{r_0}}{(N-2a-2)}+\frac {r_0^{-(a+2)}}{(a+2)(N-2a-2)}\right) r^{2+a-N}& \text{ when }\gamma>-2N.
\end{cases}\]
Therefore, we have, for $r$ large enough,
\begin{equation}
w(r)r^{N-2}\leq C\max\{r^a,r^{-2}\}.
\end{equation}

\

To estimate the function $w$ near the origin, we proceed the same way. In this case, we get
\[w(r)\leq\begin{cases}
\frac {C_{r_0}}{N-2a-2}r^{-a} +
    \frac {1}{(a-N-s)(2+a+s)}r^{2-s}\leq \left(\frac {C_{r_0}}{N-2a-2}+\d\right)r^{-a} & \text{ when }\gamma < (N-s)(s-2),\\
     \frac 1 {N-2a-2} r^{-a}\left( C_{r_0}-\log r_0 +\frac 1{N-2-2a}-\log r\right)
    & \text{ when }\gamma = (N-s)(s-2),\\
      \left( \frac{C_{r_0}}{(N-2a-2)}+\frac {r_0^{s-2-a}}{(a+2-s)(N-2a-2)}\right) r^{-a}& \text{ when }\gamma >(N-s)(s-2).
\end{cases}
\]
This proves that 
\[w(r)\leq C\max\{r^{-a},r^{2-s}\} \ \text{ as }r\to 0 .\]
\end{proof}

Thanks to the Hardy inequality, one can show the following lemma.
  \begin{lemma}\label{lem-first-eigenv}
    For every $\gamma<\left( \frac {N-2}2\right)^2$ the first eigenvalue of the operator $-\Delta- \frac\gamma {|x|^2}I$ in $D^{1,2}(\R^N)$ is strictly positive.
  \end{lemma}


\begin{thebibliography}{99}
    
    \bibitem[AG1]{AG18} {\sc A.L. Amadori, F. Gladiali}, On a singular eigenvalue problem and its applications in computing the Morse index of solutions to semilinear pde's, Nonlinear Analysis: Real World Applications {\bf 55} (2020) 103-133.
    
    \bibitem[AG2]{AG18-bif} {\sc A.L. Amadori, F. Gladiali}, non-radial sign changing solutions to Lane-Emden problem in an annulus, Nonlinear Anal. {\bf 155} (2017), 294-305.

 \bibitem[AM]{AM} {\sc A. Ambrosetti, A. Malchiodi}, Nonlinear analysis and semilinear elliptic problems,
Cambridge Studies in Advanced Mathematics, 104, Cambridge University Press, Cambridge, 2007.

\bibitem[AAP]{Ambrosetti-Azorero-Peral} {\sc A. Ambrosetti, J. Garcia Azorero, I. Peral},
Perturbation of $\Delta u+u^
\frac{N+2}{N-2}=0$, the scalar curvature problem in $\mathbb R^N$, and related topics. 
J. Funct. Anal. {\bf 165} (1999), 117-149. 

\bibitem[CGS]{Caffarelli-Gidas-Spruck}  {\sc L. Caffarelli, B. Gidas, J. Spruck}, 
Asymptotic symmetry and local behavior of semilinear elliptic equations with critical Sobolev growth, 
Comm. Pure Appl. Math. {\bf 42} (1989), no. 3, 271-297. 

\bibitem[CKN]{CKN} {\sc L. Caffarelli, R. Kohn, L. Nirenberg}, First order interpolation inequalities with weights, Compositio Math., {\bf 53} (1984),  259-275

\bibitem[CL]{Chen-Li}{\sc W. Chen,  C. Li},  Classification of solutions of some nonlinear elliptic equations. Duke Math. J. 63 (1991), no. 3, 615–622.

\bibitem[CC]{Chou-Chu}  {\sc K.S. Chou, C.W. Chu}, On the best constant for a weighted Sobolev-Hardy inequality, J. London Math. Soc. {\bf 48}, (1993), 137-151 
  
\bibitem[DX]{DX} {\sc F. Dai, Y. Xu }, {Approximation theory and harmonic analysis on spheres and balls}, Springer Monographs in Mathematics. Springer, New York, 2013. xviii+440 pp.
  
\bibitem[D]{D2}
{\sc E.N. Dancer}, {Global breaking of symmetry of positive solutions on two-dimensional annuli},  Differential Integral Equations {\bf 5}  (1992), 903-913.

\bibitem[D]{D83} {\sc E.N. Dancer}, {On the indices of fixed points of mappings in cones and applications}, Journal of Math. Anal. and Appl.  {\bf 91} (1983), 131-151.

\bibitem[DGG]{DGG}{\sc E.N. Dancer, F.~Gladiali, M.~Grossi}, On the Hardy-Sobolev equation, Proc. Roy. Soc. Edinburgh Sect. A {\bf 147} (2017), no. 2, 299--336

\bibitem[DEL16]{DEL-16} {\sc J. Dolbeault, M. Esteban, M. Loss}, Rigidity versus symmetry breaking via nonlinear flows on cylinders and Euclidean spaces. 
Invent. Math. 206 (2016), no. 2, 397--440. 

\bibitem[DEL17]{DEL-17} {\sc J. Dolbeault, M. Esteban, M. Loss}, Symmetry and symmetry breaking: rigidity and flows in elliptic PDE's, International Congress of Mathematicians, 2018, Rio de Janeiro,
Brazil. 3, pp.2279-2304, 2017, Proc. Int. Cong. of Math., Rio de Janeiro. 

\bibitem[DELT]{DELT09} {\sc J. Dolbeault, M. Esteban, M. Loss, G. Tarantello}, On the symmetry of extremals for the {C}affarelli-{K}ohn-{N}irenberg inequalities,
Advanced Nonlinear Studies, {\bf 9} (2009), 713-726

\bibitem[FS]{FS} {\sc V. Felli, M. Schneider}, Perturbation results of critical elliptic equations of {C}affa\-relli-{K}ohn-{N}irenberg type, J. Differential Equations,  {\bf 191} (2003), 121--142
 
\bibitem[GR1]{GhRo3} {\sc N. Ghoussoub, F. Robert}, 
The effect of curvature on the best constant in the Hardy-Sobolev inequalities, Geom. Funct. Anal. {\bf 16} (2006), no. 6, 1201-1245.

\bibitem[GR2]{GhRo4} {\sc N. Ghoussoub, F. Robert}, 
 Sobolev inequalities for the Hardy-Schrodinger operator: extremals and critical dimensions, Bull. Math.
Sci. {\bf 6} (2016), no. 1, 89-144.

\bibitem[GR3]{GR-interior} {\sc N. Ghoussoub, F. Robert}, 
The Hardy-Schrodinger operator with interior singularity: the remaining cases, 
Calc. Var. Partial Differential Equations {\bf 56} (2017), 56-149. 

\bibitem[GR4]{GhRo} {\sc N. Ghoussoub, F. Robert}, 
On the Hardy-Schrodinger operator with a boundary singularity,
Analysis and PDEs, {\bf 10-5}, (2017), 1017-1079. 


\bibitem[GS]{Gidas-Spruck81}{\sc B. Gidas, J. Spruck},  
Global and local behavior of positive solutions of nonlinear elliptic equations,
Comm. Pure Appl. Math. {\bf 34} (1981), no. 4, 525-598. 

 \bibitem[G]{G10}
{\sc F. Gladiali}, {A global bifurcation result for a semilinear elliptic equation}, Journal of Math. Anal. and Appl.
{\bf 369} (2010),  306-311.

\bibitem[G2]{G}{\sc F.~Gladiali}, Separation of branches of $O(N-1)$-invariant solutions for a semilinear elliptic equation, Journal of Mathematical Analysis and Applications {\bf 453}, (2017), 159-173
  
\bibitem[GPW]{GPW}{\sc F.~Gladiali, F.~Pacella, T. Weth}, Symmetry and nonexistence of low Morse index solutions in unbounded domains. J. Math. Pures Appl. {\bf  93} (2010), 536-558.

\bibitem[GGT]{GGT2}{\sc F.~Gladiali, M.~Grossi, C. Troestler}, Entire radial and non-radial solutions for systems with critical growth. Calc. Var. Partial Differential Equations {\bf 57} (2018), 

\bibitem[GMGT]{GMGT} {\sc V. Glaser, A. Martin, H. Grosse, W. Thirring,} A Family of Optimal Conditions for the Absence of Bound States in a Potential. Les rencontres physiciens-math\'ematiciens de Strasbourg, RCP25, Tome 23 (1976) Expos\'e no. 1, 22 p. \url{http://www.numdam.org/item/RCP25_1976__23__0_0/}

\bibitem[Il]{Il} {\sc Il'in, V. P.}, Some integral inequalities and their applications in the theory of differentiable functions of several variables, Mat. Sb. {\bf 54} (1961), 331-380

\bibitem[JLX]{JLX} {\sc Q. Jin, Y. Li and H. Xu}, Symmetry and asymmetry: the method of moving spheres. Adv. Differential Equations {\bf 13} (2008),  601–640.

\bibitem[LZ]{YanYan-Li03}
{\sc Y.Y. Li, L. Zhang}, Liouville-type theorems and harnack-type inequalities for semilinear
  elliptic equations. Journal d'Analyse Mathématique {\bf 90} (2003), 27-87.

\bibitem[MW]{MW} {\sc M. Musso, J. Wei}, non-radial solutions to critical elliptic equations of Caffarelli-Kohn-Nirenberg type, Int. Math. Res. Not. {\bf 2012} (2012), 4120-4162.

  \bibitem[N]{Ni}{\sc W. M. Ni}, A Nonlinear Dirichlet Problem on the Unit Ball and Its Applications, {Indiana Univ. Math. J.} {\bf 31}, (1982), 801-807.

\bibitem[Ra]{Rabinowitz} {\sc P.H. Rabinowitz},Some global results for nonlinear eigenvalue problems. 
J. Functional Analysis {\bf 7} (1971), 487-513. 

  \bibitem[Ro]{Robert} {\sc F. Robert},Nondegeneracy of positive solutions to nonlinear Hardy-Sobolev equations. Adv. Nonlinear Anal. {\bf 6}, (2017), 237-242.

  \bibitem[S]{Smets} {\sc D. Smets},Nonlinear Schrodinger equations with Hardy potential and critical nonlinearities, Trans. Amer. Math. Soc. {\bf 357}, 2909-2938.
    
   \bibitem[SW]{SW} {\sc J. Smoller, A. Wasserman}, Bifurcation and symmetry-breaking, {Invent. Math.} {\bf 100}, (1990), 63-95.

\bibitem[Ta]{Talenti} {\sc G. Talenti},Best constant in Sobolev inequality, 
Ann. Mat. Pura Appl. (4) {\bf 110},  (1976), 353-372.
 
\bibitem[Te]{T} {\sc S. Terracini},On positive entire solutions to a class of equations with a singular
coefficient and critical exponent, {Adv. Diff. Eq.} {\bf 1}, (1996), 241-264.



\end{thebibliography}
\end{document}